\RequirePackage{scrlfile}
\BeforeClass{tac}{\usepackage{amsmath}}
\documentclass{tac}

\usepackage[english,strings]{babel}
\usepackage[utf8]{inputenc}
\usepackage[T1]{fontenc}
\usepackage{amssymb}
\usepackage{amscd}
\usepackage{tikz}
\usetikzlibrary{cd}
\usepackage[all,cmtip]{xy}
\usepackage{authblk}
\usepackage{enumitem}

\DeclareMathOperator{\Image}{Im}
\DeclareRobustCommand\i{i}
\DeclareRobustCommand\ip{i+1}
\DeclareRobustCommand\j{j}
\DeclareRobustCommand\jp{j+1}

\newcommand{\eqrefi}[2]{\renewcommand\i{#2}\eqref{#1}\renewcommand\i{i}}
\newcommand{\eqrefip}[2]{\renewcommand\ip{#2}\eqref{#1}\renewcommand\ip{i+1}}
\newcommand{\eqrefj}[2]{\renewcommand\j{#2}\eqref{#1}\renewcommand\j{j}}
\newcommand{\eqrefjp}[2]{\renewcommand\jp{#2}\eqref{#1}\renewcommand\jp{j+1}}
\newcommand{\eqrefij}[3]{\renewcommand\i{#2}\renewcommand\j{#3}\eqref{#1}\renewcommand\i{i}\renewcommand\j{j}}
\newcommand{\eqrefijp}[3]{\renewcommand\i{#2}\renewcommand\jp{#3}\eqref{#1}\renewcommand\i{i}\renewcommand\jp{j+1}}

\newdir{ >}{@{}*!/-10pt/@{>}} 

\title{A syntactic characterization of weakly Mal'tsev varieties}

\author[a]{Nadja Egner}
\author[a,b]{Pierre-Alain Jacqmin}
\author[c]{Nelson Martins-Ferreira}
\affil[a]{\small{\textit{Institut de Recherche en Math\'ematique et Physique, Universit\'e catholique de Louvain, Louvain-la-Neuve, Belgium\vspace{5pt}}}}
\affil[b]{\small{\textit{Department of Mathematics, Royal Military Academy, Brussels, Belgium\vspace{5pt}}}}
\affil[c]{\small{\textit{Instituto Polit\'ecnico de Leiria, Leiria, Portugal\vspace{5pt}}}}
\affil[ ]{Email: nadja.egner@uclouvain.be, pierre-alain.jacqmin@uclouvain.be, martins.ferreira@ipleiria.pt}

\copyrightyear{2023}

\cauthor{Nadja Egner, Pierre-Alain Jacqmin, and Nelson Martins-Ferreira}

\keywords{weakly Mal'tsev category, weakly Mal'tsev variety, Mal'tsev condition, syntactic characterization, strong relation, pullback injection}
\amsclass{08B05, 18E13 (primary); 18C10, 06B20, 06D99, 18D40 (secondary)}

\dedication{In memory of Pieter Hofstra}

\begin{document}

\maketitle

\begin{abstract}
The notion of a weakly Mal'tsev category, as it was introduced in 2008 by the third author, is a generalization of the classical notion of a Mal'tsev category. It is well-known that a variety of universal algebras is a Mal'tsev category if and only if its theory admits a Mal'tsev term. In the main theorem of this paper, we prove a syntactic characterization of the varieties that are weakly Mal'tsev categories. We apply our result to the variety of distributive lattices which was known to be a weakly Mal'tsev category before. By a result of Z.~Janelidze and the third author, a finitely complete category is weakly Mal'tsev if and only if any internal strong reflexive relation is an equivalence relation. In the last part of this paper, we give a syntactic characterization of those varieties in which any regular reflexive relation is an equivalence relation.
\end{abstract}

\section{Introduction}

The study of Mal'tsev categories originates with Mal'tsev's paper~\cite{mal'tsev:1954} from 1954 where he showed that, for a variety $\mathbb{V}$ of (finitary one-sorted) universal algebras, the composition of congruences on a fixed algebra is commutative if and only if the algebraic theory of $\mathbb{V}$ contains a ternary term $p(x,y,z)$ such that the two identities $p(x,x,y)=y$ and $p(x,y,y)=x$ are satisfied. Such varieties are nowadays called Mal'tsev varieties (or $2$-permutable varieties) and a term $p(x,y,z)$ as described above a Mal'tsev term. Examples of Mal'tsev varieties are given by the varieties of groups, of rings, of Lie algebras and of Heyting algebras. In~\cite{carboni.lambek.pedicchio:1990} from 1990, Carboni, Lambek and Pedicchio introduced the notion of a Mal'tsev category in the context of (Barr-)exact categories via the former condition, and developed some aspects of non-abelian homological algebra. In this setting, the commutativity of the composition of internal equivalence relations on a fixed object is equivalent to any reflexive relation being an equivalence relation or any relation being difunctional. In~\cite{carboni.pedicchio.pirovano:1992} from 1992, Carboni, Pedicchio and Pirovano defined Mal'tsev categories in the finitely complete setting via the latter two, still equivalent, conditions. In addition to the varietal examples given above, one can cite as examples of Mal'tsev categories the category of topological groups, any abelian category and the dual of any elementary topos. We refer the reader to~\cite{bourn.gran.jacqmin:2021} for further details on the history of the development of Mal'tsev categories. 

Mal'tsev categories turned out to be a central concept in categorical algebra, especially in the development of the notion of centrality of equivalence relations~\cite{pedicchio:1996, bourn.gran:2002}. Further results in Mal'tsev categories were proved in the study of central extensions~\cite{janelidze.kelly:1994, everaert:2014, duvieusart.gran:2018} and homological lemmas such as the denormalized $3\times 3$-lemma~\cite{bourn:2003} which is, in the regular context, equivalent to the weaker Goursat property (also known as $3$-permutability). More recently, some embedding theorems have been established for Mal'tsev categories~\cite{jacqmin:2018, jacqmin:2019}, similar to the Freyd-Mitchell embedding theorem for abelian categories.

Weakly Mal'tsev categories were introduced by the third author in \cite{martins-ferreira:2008} as a generalization of the notion of a Mal'tsev category. It was shown in \cite{bourn:1996} that a Mal'tsev category is exactly a finitely complete category $\mathfrak{C}$ such that for any pullback diagram
\begin{equation}\label{eq:pullbackofsplitepimorphisms}
	\begin{tikzcd}
		{X\times_Z Y} && Y \\
		\\
		X && Z
		\arrow["f"', shift right=1, from=3-1, to=3-3]
		\arrow["r"', shift right=1, from=3-3, to=3-1]
		\arrow["g", shift left=1, from=1-3, to=3-3]
		\arrow["s", shift left=1, from=3-3, to=1-3]
		\arrow["{p_1}"', shift right=1, from=1-1, to=3-1]
		\arrow["{e_1}"', shift right=1, from=3-1, to=1-1]
		\arrow["{p_2}", shift left=1, from=1-1, to=1-3]
		\arrow["{e_2}", shift left=1, from=1-3, to=1-1]
		\arrow["\scalebox{2}{$ \lrcorner $}"{anchor=center, pos=0.125}, draw=none, from=1-1, to=3-3]
	\end{tikzcd}
\end{equation}
in~$\mathfrak{C}$, where $f$ and $g$ are two split epimorphisms with respective splittings $r$ and~$s$, and $p_1$ and $p_2$ are the pullback projections, the canonical pullback injections $e_1$ and $e_2$ induced by $r$ and~$s$, respectively, are jointly strongly epimorphic. In the finitely complete context, this is equivalent to say that $e_1$ and $e_2$ are jointly extremally epimorphic. In~\cite{martins-ferreira:2008}, a category $\mathfrak{C}$ is called weakly Mal'tsev if it admits pullbacks of split epimorphisms along split epimorphisms, and the pullback injections $e_1$ and $e_2$ in a diagram as \eqref{eq:pullbackofsplitepimorphisms} are jointly epimorphic. A direct consequence of this definition is that a reflexive graph
\begin{equation*}
	\begin{tikzcd}
		C_1 \arrow[rr, "c"', shift right=3] \arrow[rr, "d", shift left=3] && C_0 \arrow[ll, "e" description]
	\end{tikzcd}
\end{equation*}
in a weakly Mal'tsev category admits at most one multiplicative graph structure. This means that there exists at most one "composition" map $m\colon C_2\to C_1$, where $C_2$ is the pullback
\begin{equation*}
	\begin{tikzcd}
		{C_2} && C_1 \\
		\\
		C_1 && C_0
		\arrow["d"', shift right=1, from=3-1, to=3-3]
		\arrow["e"', shift right=1, from=3-3, to=3-1]
		\arrow["c", shift left=1, from=1-3, to=3-3]
		\arrow["e", shift left=1, from=3-3, to=1-3]
		\arrow["{\overline{p_1}}"', shift right=1, from=1-1, to=3-1]
		\arrow["{\overline{e_1}}"', shift right=1, from=3-1, to=1-1]
		\arrow["{\overline{p_2}}", shift left=1, from=1-1, to=1-3]
		\arrow["{\overline{e_2}}", shift left=1, from=1-3, to=1-1]
		\arrow["\scalebox{2}{$ \lrcorner $}"{anchor=center, pos=0.125}, draw=none, from=1-1, to=3-3]
	\end{tikzcd}
\end{equation*}
of $d$ along~$c$, such that 
\begin{equation*}
	m\overline{e_1}=1_{C_1}=m\overline{e_2}
\end{equation*}
holds, where $\overline{e_1}$, $\overline{e_2}$ are the pullback injections induced by the common splitting $e$ of $d$ and~$c$. Furthermore, one can show that every multiplicative graph in a weakly Mal'tsev category yields automatically an internal category, i.e., $m$ satisfies the usual identity and associativity axioms internally. However, an internal category in a weakly Mal'tsev category can fail to yield an internal groupoid as it is the case for Mal'tsev categories \cite{martins-ferreira:2008,martins-ferreira.vanderlinden:2014}. This means that there are internal categories 
\begin{equation*}
	\begin{tikzcd}
		C_2 \arrow[rr, "m"] && C_1 \arrow[rr, shift left=3, "d"] \arrow[rr, shift right=3, "c"'] && C_0 \arrow[ll, "e" description]
	\end{tikzcd}
\end{equation*}
in certain weakly Mal'tsev categories that do not allow for an "inverse" map $i\colon C_1\to C_1$.

A (finitary one-sorted) variety $\mathbb{V}$ of universal algebras is a Mal'tsev category if and only if its theory contains a ternary term $p(x,y,z)$ such that the equations $p(x,x,y)=y$ and $p(x,y,y)=x$ are satisfied in~$\mathbb{V}$~\cite{mal'tsev:1954,mal'tsev:1963}. Surprisingly, no similar syntactic characterization of weakly Mal'tsev varieties was proved. The main purpose of this paper is to establish such a syntactic characterization (Theorem~\ref{thm:syntaxWM}). In contrast to Mal'tsev varieties where we have one ternary term $p(x,y,z)$ which fulfills certain identities, we get that a variety $\mathbb{V}$ is a weakly Mal'tsev category if and only if there exist integers $k,m,N\geqslant 0$, binary terms $f_1,g_1,\ldots,f_k,g_k$, ternary terms $p_1,\ldots,p_m$, $(2(k+2m+1))$-ary terms $s_1,\ldots,s_N$, $(2(k+m+2))$-ary terms $\sigma_1,\ldots,\sigma_{N+1}$ and, for all $i\in\{1,\ldots,N+1\}$, $(k+m+1)$-ary terms $\eta_1^{(i)}, \eta_2^{(i)}, \epsilon_1^{(i)}, \epsilon_2^{(i)}$ that satisfy certain equations. This phenomenon of having the number of terms or their arities not being fixed in a syntactic characterization also occurs for congruence distributive~\cite{jonsson:1967}, congruence modular~\cite{day:1969,gumm:1981} and protomodular varieties~\cite{bourn.janelidze:2003}.

Two (quasi-)algebraic examples of weakly Mal'tsev categories that are not Mal'tsev are given by the category of commutative monoids with cancellation~\cite{martins-ferreira:2008} and the category of distributive lattices~\cite{martins-ferreira:2012}, see also \cite{martins-ferreira:2015} for examples of co-weakly Mal'tsev categories, i.e., categories whose dual category is weakly Mal'tsev. We will apply our main result to the variety of distributive lattices.

Let us briefly describe the strategy we used to find the syntactic characterization of weakly Mal'tsev varieties. The first step was to look for the right formulation of the property of being a weakly Mal'tsev category and apply it to the right diagram made from free algebras in the variety. In order to do so, we expressed the property, in the finitely complete and cocomplete context, as the property that for each pair of split epimorphisms $f$ and $g$ with common codomain and respective sections $r$ and~$s$, considering the pullback diagram~\eqref{eq:pullbackofsplitepimorphisms} and the cokernel pair
\begin{equation*}
	\begin{tikzcd}
		X+Y && X\times_Z Y && Q
		\arrow["{[e_1,e_2]}", from=1-1, to=1-3]
		\arrow["{q_1}", shift left=1, from=1-3, to=1-5]
		\arrow["{q_2}"', shift right=1, from=1-3, to=1-5]
	\end{tikzcd}
\end{equation*}
of the induced morphism $[e_1,e_2]$ from the coproduct $X+Y$, one has $q_1=q_2$. This property is of the type studied in~\cite{jacqmin:2022} (generalizing in some context the type of properties studied in~\cite{jacqmin.janelidze:2021} of which the Mal'tsev property is an example). Using the results from~\cite{jacqmin:2022}, we immediately get that, for a variety~$\mathbb{V}$, it is equivalent to only consider the particular case of the pullback
\begin{equation*}
	\begin{tikzcd}
		P && {\mathsf{F}(x,y)} \\
		\\
		{\mathsf{F}(x,y)} && {\mathsf{F}(x),}
		\arrow["r"', shift right=1, from=3-3, to=3-1]
		\arrow["f"', shift right=1, from=3-1, to=3-3]
		\arrow["f", shift left=1, from=1-3, to=3-3]
		\arrow["s", shift left=1, from=3-3, to=1-3]
		\arrow["{e_1}"', shift right=1, from=3-1, to=1-1]
		\arrow["{p_1}"', shift right=1, from=1-1, to=3-1]
		\arrow["{p_2}", shift left=1, from=1-1, to=1-3]
		\arrow["{e_2}", shift left=1, from=1-3, to=1-1]
		\arrow["\scalebox{2}{$\lrcorner$}"{anchor=center, pos=0.125}, draw=none, from=1-1, to=3-3]
	\end{tikzcd}
\end{equation*} 
where $f$ is the unique morphism from the free algebra on two generators to the free algebra on one generator such that $f(x)=x=f(y)$, $r$ is the unique morphism such that $r(x)=x$ and $s$ is the unique morphism such that $s(x)=y$. Furthermore, Theorem~4.1 in~\cite{jacqmin:2022} tells us that the variety $\mathbb{V}$ is weakly Mal'tsev if and only if $q_1(y,x)=q_2(y,x)$, where $(q_1,q_2)$ is the cokernel pair of $[e_1,e_2]\colon \mathsf{F}(x,y)+\mathsf{F}(x,y)\to P$
\begin{equation*}
	\begin{tikzcd}
		{\mathsf{F}(x,y)+\mathsf{F}(x,y)} && P && Q
		\arrow["{[e_1,e_2]}", from=1-1, to=1-3]
		\arrow["{q_1}", shift left=1, from=1-3, to=1-5]
		\arrow["{q_2}"', shift right=1, from=1-3, to=1-5]
	\end{tikzcd}
\end{equation*}
and where $(y,x)\in P$ is the unique element such that $p_1(y,x)=y$ and $p_2(y,x)=x$. For the sake of completeness, we reprove this result here for this specific property instead of applying the results of~\cite{jacqmin:2022} (Lemma~\ref{lem:WMif and only ifprojection}).

The second part of the proof consists in identifying the terms whose existence is equivalent to the equality $q_1(y,x)=q_2(y,x)$. In order to do so, we use the description of $q_1$ and $q_2$ by means of the coequalizer $q$ of the two maps $\iota_1[e_1,e_2]$ and $\iota_2[e_1,e_2]$, where $\iota_1,\iota_2$ are the two coproduct inclusions from $P$ to $P+P$:
\begin{equation*}
	\xymatrix{\mathsf{F}(x,y)+\mathsf{F}(x,y) \ar[rr]^-{[e_1,e_2]} \ar[dd]_-{[e_1,e_2]} && P \ar[dd]^-{\iota_2} \ar@/^1.2pc/[rddd]^-{q_2} \\
		&& \\
		P \ar[rr]_-{\iota_1} \ar@/_1.2pc/[rrrd]_{q_1} && {P+P} \ar[rd]^-{q} \\
		&&& Q}
\end{equation*}
We thus think of $Q$ as a quotient of $P+P$ which itself can be constructed as a quotient of the free algebra $\mathsf{F}(\mathsf{U}P+\mathsf{U}P)$ on the disjoint union of two copies of the underlying set $\mathsf{U}P$ of~$P$.

The term condition we obtain by `brute-force description' of the equality $q_1(y,x)=q_2(y,x)$ is unfortunately very long and complex. The final step of the proof is to simplify this characterization in order to get an equivalent formulation of it which is easier to get intuition of.

The paper is structured as follows. In Section~\ref{sect:Preliminaries}, we recall the necessary material for our main theorem. In particular, we recall the description of the coproduct of two non-empty algebras $A$ and $B$ in a variety $\mathbb{V}$ of universal algebras (Section~\ref{sect:Coproductoftwoalgebras}) and some characterizations of Mal'tsev and weakly Mal'tsev categories in the finitely complete context (Section~\ref{sect:(Weakly)Mal'tsevcategories}). Section~\ref{sect:WeaklyMal'tsevvarieties} proves our syntactic characterization of weakly Mal'tsev varieties (Theorem~\ref{thm:syntaxWM}). Furthermore, in Example~\ref{ex:distributivelattices}, we give the weakly Mal'tsev terms for the variety of distributive lattices. Finally, we show that a slight variation of Theorem~\ref{thm:syntaxWM} yields a syntactic characterization of the varieties in which any reflexive regular relation is an equivalence relation (Theorem~\ref{thm:syntaxwM}).

\underline{\textbf{Terminology:}} In this paper, all varieties of universal algebras are understood to be finitary and one-sorted.

\subsection*{Acknowledgements}

The authors would like to thank Marino Gran and George Ja\-ne\-lidze for their interesting comments on an earlier version of the paper. They also would like to thank the anonymous referee for their remarks that helped to improve the readability of the paper.

The first and second authors are grateful to the FNRS for its support. The first author also thanks UCLouvain for the FSR grant she received during the research leading to this article. The third author was funded by FCT/MCTES (PIDDAC): UIDP/04044/2020; Generative Thermodynamic; Associate Laboratory ARISE LA/P/0112/2020; MATIS (CENTRO-01-(0145, 0247)-FEDER-(000014, 069665, 039969, 003362)); POCI-01-0247-FEDER-(069603, 039958, 039863, 024533); by CDRSP and ESTG from the Polytechnic of Leiria.

\section{Preliminaries}\label{sect:Preliminaries}

\subsection{Coproduct of two algebras}\label{sect:Coproductoftwoalgebras}

As we will need it in the proof of our main result, we recall a description of the coproduct of two non-empty algebras $A$ and $B$ in a variety $\mathbb{V}$ of universal algebras. We denote by $\mathsf{U}\colon\mathbb{V}\to\mathsf{Set}$ the forgetful functor and $\mathsf{F}\colon\mathsf{Set}\to\mathbb{V}$ the free functor.
Then the coproduct $A+B$ is given by the quotient of the free algebra $\mathsf{F}(\mathsf{U}A+\mathsf{U}B)$ on the disjoint union of the underlying sets of $A$ and $B$ with respect to the smallest congruence $C$ that turns the two set-theoretic functions $i_1\colon A\to \mathsf{F}(\mathsf{U}A+\mathsf{U}B)$ and $i_2\colon B\to \mathsf{F}(\mathsf{U}A+\mathsf{U}B)$, that interpret elements of $A$ and $B$ as variables, into morphisms in~$\mathbb{V}$. More precisely, $C$ is the congruence generated by all pairs of the form $(i_1(\omega^A(a_1,\ldots,a_k)),\omega(i_1(a_1),\ldots,i_1(a_k)))$ or $(i_2(\omega^B(b_1,\ldots,b_k)),\omega(i_2(b_1),\ldots,i_2(b_k)))$, where $k\geqslant 0$ is an integer, $a_1,\ldots,a_k\in A$ and $b_1,\ldots,b_k\in B$ are elements, and $\omega$ is a $k$-ary operation of~$\mathbb{V}$. Here $\omega^A$ represents the realization of $\omega$ as a function $\omega^A\colon A^k\to A$ and similarly for~$\omega^B$. In the following, we will omit the functions $i_1$ and~$i_2$. Given integers $k,\ell\geqslant 0$, elements $a_1,\ldots,a_k\in A$ and $b_1,\ldots,b_{\ell}\in B$, and a $(k+\ell)$-ary term~$s$, we have that
\begin{equation}\label{eq:coproduct}
	s(\iota_1(a_1),\ldots,\iota_1(a_k),\iota_2(b_1),\ldots,\iota_2(b_{\ell}))
	=[s(a_1,\ldots,a_k,b_1,\ldots,b_{\ell})],
\end{equation}
where $\iota_1\colon A\to\mathsf{F}(\mathsf{U}A+\mathsf{U}B)/C$ and $\iota_2\colon B\to \mathsf{F}(\mathsf{U}A+\mathsf{U}B)/C$ are the coproduct inclusions induced by $i_1$ and~$i_2$ respectively, and $[\cdot]$ denotes the equivalence class of an element in $\mathsf{F}(\mathsf{U}A+\mathsf{U}B)$ with respect to~$C$. Given another algebra~$D$, and morphisms $f\colon A\to D$ and $g\colon B\to D$, the unique morphism $\varphi\colon\mathsf{F}(\mathsf{U}A+\mathsf{U}B)/C\to D$ such that $\varphi\iota_1=f$ and $\varphi\iota_2=g$ is given by
\begin{equation*}
	\varphi([s(a_1,\ldots,a_k,b_1,\ldots,b_\ell)]):=s^D(f(a_1),\ldots,f(a_k),g(b_1),\ldots,g(b_\ell)).
\end{equation*}
Alternatively, we consider the congruence $C'$ on $\mathsf{F}(\mathsf{U}A+\mathsf{U}B)$ generated by the relation $\sim$ given by all pairs of the form
\begin{multline*}
	(\tau(a_1,\ldots,a_m,b_1,\ldots,b_{n},\mu_1(a_1,\ldots,a_m),\mu_2(b_1,\ldots,b_{n})),\\
	\tau(a_1,\ldots,a_m,b_1,\ldots,b_{n},\lambda_1(a_1,\ldots,a_m),\lambda_2(b_1,\ldots,b_{n}))),
\end{multline*}
where $m,n\geqslant 0$ are integers, $a_1,\ldots,a_m\in A$ and $b_1\ldots,b_n\in B$ are elements, $\mu_1,\lambda_1$ are $m$-ary terms, $\mu_2,\lambda_2$ are $n$-ary terms and $\tau$ is an $(m+n+2)$-ary term with
\begin{align*}
	\mu^A_1(a_1,\ldots,a_m)&=\lambda^A_1(a_1,\ldots,a_m),\\
	\mu^B_2(b_1,\ldots,b_{n})&=\lambda^B_2(b_1,\ldots,b_{n}).
\end{align*}
Since $A$ and $B$ are supposed to be non-empty, the relation $\sim$ is reflexive. Furthermore, it is symmetric. We state the following classical results.

\begin{proposition}\label{prop:coproduct}
	Let $A$ and $B$ be two non-empty algebras in~$\mathbb{V}$.
	\begin{enumerate}
		\item The two congruences $C$ and $C'$ on $\mathsf{F}(\mathsf{U}A+\mathsf{U}B)$ are equal.
		\item The congruence $C'$ is given by the transitive closure of~$\sim$.
	\end{enumerate}
\end{proposition}

\subsection{(Weakly) Mal'tsev categories}\label{sect:(Weakly)Mal'tsevcategories}

Let us recall some characterizations of Mal'tsev categories in the finitely complete context.

\begin{definition}[Mal'tsev category~\cite{carboni.pedicchio.pirovano:1992}]\label{def:Mal'tsevcategory}
Let $\mathfrak{C}$ be a finitely complete category. We call $\mathfrak{C}$ a \emph{Mal'tsev category} if any reflexive relation $r\colon R\rightarrowtail X\times X$ in $\mathfrak{C}$ is an equivalence relation.
\end{definition}

We recall that a relation $R\subseteq X\times Y$ between two sets $X$ and $Y$ is called \emph{difunctional}~\cite{riguet:1948} if, for given elements $x,x'\in X$ and $y,y'\in Y$, the conditions $xRy$, $x'Ry$, $x'Ry'$ imply that $xRy'$. This can be immediately generalized to relations in an arbitrary category using generalized elements.

\begin{proposition}\cite{carboni.pedicchio.pirovano:1992}\label{prop:characterizationMal'tsevcategoriesrelations}
Let $\mathfrak{C}$ be a finitely complete category. Then the following conditions are equivalent:
\begin{enumerate}
	\item The category $\mathfrak{C}$ is a Mal'tsev category. 
	\item Any reflexive relation $r\colon R\rightarrowtail X\times X$ in $\mathfrak{C}$ is symmetric.
	\item Any reflexive relation $r\colon R\rightarrowtail X\times X$ in $\mathfrak{C}$ is transitive.
	\item Any relation $r\colon R\rightarrowtail X\times X$ in $\mathfrak{C}$ is difunctional.
\end{enumerate}
\end{proposition}

In the following, we will also use another characterization of those finitely complete categories which are Mal'tsev categories. Let $f\colon X\to Z$ and $g\colon Y\to Z$ be two split epimorphisms in a finitely complete category $\mathfrak{C}$ with splittings $r\colon Z\to X$ and $s\colon Z\to Y$ respectively, i.e., $fr=1_Z=gs$. We consider the pullback
\begin{equation}\label{diagr:genericpullbackofsplitepis}
	\begin{tikzcd}
		{X\times_Z Y} && Y \\
		\\
		X && Z
		\arrow["f"', shift right=1, from=3-1, to=3-3]
		\arrow["r"', shift right=1, from=3-3, to=3-1]
		\arrow["g", shift left=1, from=1-3, to=3-3]
		\arrow["s", shift left=1, from=3-3, to=1-3]
		\arrow["{p_1}"', shift right=1, from=1-1, to=3-1]
		\arrow["{e_1}"', shift right=1, dashed, from=3-1, to=1-1]
		\arrow["{p_2}", shift left=1, from=1-1, to=1-3]
		\arrow["{e_2}", shift left=1, dashed, from=1-3, to=1-1]
		\arrow["\scalebox{2}{$\lrcorner$}"{anchor=center, pos=0.125}, draw=none, from=1-1, to=3-3]
	\end{tikzcd}
\end{equation}
of $f$ along~$g$. Let $e_1\colon X\to X\times_Z Y$ and $e_2\colon Y\to X\times_Z Y$ be the unique morphisms such that
\begin{align*}
	&p_1 e_1=1_X, 
	&p_2 e_1=sf,\\
	&p_1 e_2=rg,
	&p_2 e_2=1_Y.
\end{align*} 

\begin{proposition}\cite{bourn:1996}\label{prop:characterizationMal'tsevcategoriespullbacks}
Let $\mathfrak{C}$ be a finitely complete category. Then the following conditions are equivalent:
\begin{enumerate}
	\item The category $\mathfrak{C}$ is a Mal'tsev category.
	\item For any Diagram~\eqref{diagr:genericpullbackofsplitepis} in~$\mathfrak{C}$, the morphisms $e_1\colon X\to X\times_Z Y$ and $e_2\colon Y\to X\times_Z Y$ are jointly extremally epimorphic, i.e., if there exist an object $M$ and morphisms $f_1\colon X\to M$, $f_2\colon Y\to M$ and a monomorphism $m\colon M\rightarrowtail X\times_Z Y$ such that $m f_i=e_i$ for each $i\in\{1,2\}$, then $m$ is an isomorphism.
\end{enumerate}
\end{proposition}

A direct generalization of this characterization of Mal'tsev categories gives rise to the notion of a weakly Mal'tsev category.

\begin{definition}[Weakly Mal'tsev category~\cite{martins-ferreira:2008}]\label{def:weaklyMal'tsevcategory}
Let $\mathfrak{C}$ be a category. We call $\mathfrak{C}$ a \emph{weakly Mal'tsev category} if the following conditions hold:
\begin{enumerate}
	\item The category $\mathfrak{C}$ possesses all pullbacks of split epimorphisms along split epimorphisms.
	\item For any Diagram~\eqref{diagr:genericpullbackofsplitepis} in~$\mathfrak{C}$, the morphisms $e_1\colon X\to X\times_Z Y$ and $e_2\colon Y\to X\times_Z Y$ are jointly epimorphic, i.e, given an object $A$ and morphisms $u,v\colon X\times_Z Y \to A$ such that $ue_1=ve_1$ and $ue_2=ve_2$, then $u=v$.
\end{enumerate}
\end{definition}

We call a relation $r\colon R\rightarrowtail X\times Y$ \emph{strong} if it is a strong monomorphism, i.e., for any epimorphism $e\colon A\twoheadrightarrow B$ and morphisms $a\colon A\to R$ and $b\colon B\to X\times Y$ such that the outer part of the diagram
\begin{equation*}
	\begin{tikzcd}
		A && B \\
		\\
		R && {X\times Y}
		\arrow["r"', tail, from=3-1, to=3-3]
		\arrow["e", two heads, from=1-1, to=1-3]
		\arrow["a"', from=1-1, to=3-1]
		\arrow["b", from=1-3, to=3-3]
		\arrow["d", dashed, from=1-3, to=3-1]
	\end{tikzcd}
\end{equation*}
commutes, there exists a unique morphism $d\colon B\to R$ such that $de=a$ and $rd=b$.

\begin{proposition}\cite{janelidze.martins-ferreira:2012}\label{prop:characterizationweaklyMal'tsevcategories}
Let $\mathfrak{C}$ be a finitely complete category. Then the following conditions are equivalent:
\begin{enumerate}
	\item The category $\mathfrak{C}$ is a weakly Mal'tsev category.
	\item Any reflexive strong relation $r\colon R\rightarrowtail X\times X$ in $\mathfrak{C}$ is an equivalence relation.
	\item Any reflexive strong relation $r\colon R\rightarrowtail X\times X$ in $\mathfrak{C}$ is symmetric.
	\item Any reflexive strong relation $r\colon R\rightarrowtail X\times X$ in $\mathfrak{C}$ is transitive.
	\item Any strong relation $r\colon R\rightarrowtail X\times X$ in $\mathfrak{C}$ is difunctional.
\end{enumerate}
\end{proposition}

\section{Weakly Mal'tsev varieties}\label{sect:WeaklyMal'tsevvarieties}

Our aim is to find a syntactic characterization of the varieties which are weakly Mal'tsev categories. Let us recall the classical theorem of Mal'tsev which identifies all the varieties that are Mal'tsev categories.

\begin{theorem}[Mal'tsev's theorem~\cite{mal'tsev:1954,mal'tsev:1963}]\label{thm:Mal'tsev}
Let $\mathbb{V}$ be a variety. The category $\mathbb{V}$ is Mal'tsev if and only if there exists a ternary term $p\in\mathsf{F}(x,y,z)$ such that $p(x,x,y)=y$ and $p(x,y,y)=x$.
\end{theorem}

We recall a proof of Mal'tsev's theorem which relies on the characterization of Mal'tsev categories as those finitely complete categories in which the canonical inclusions into a pullback of split epimorphisms are jointly extremally epimorphic that we recalled in Proposition~\ref{prop:characterizationMal'tsevcategoriespullbacks}. We will then adapt it for the case of weakly Mal'tsev categories. Let us consider the pullback
\begin{equation*}
	\begin{tikzcd}
		P && {\mathsf{F}(x,y)} \\
		\\
		{\mathsf{F}(x,y)} && {\mathsf{F}(x),}
		\arrow["r"', shift right=1, from=3-3, to=3-1]
		\arrow["f"', shift right=1, from=3-1, to=3-3]
		\arrow["f", shift left=1, from=1-3, to=3-3]
		\arrow["s", shift left=1, from=3-3, to=1-3]
		\arrow["{e_1}"', shift right=1, from=3-1, to=1-1]
		\arrow["{p_1}"', shift right=1, from=1-1, to=3-1]
		\arrow["{p_2}", shift left=1, from=1-1, to=1-3]
		\arrow["{e_2}", shift left=1, from=1-3, to=1-1]
		\arrow["\scalebox{2}{$\lrcorner$}"{anchor=center, pos=0.125}, draw=none, from=1-1, to=3-3]
	\end{tikzcd}
\end{equation*}
where $\mathsf{F}(x)$ and $\mathsf{F}(x,y)$ are the free algebras on the variables $x$ and $x,y$ respectively, $f$ is the unique morphism such that $f(x)=x=f(y)$, $r$ is the unique morphism such that $r(x)=x$ and $s$ is the unique morphism such that $s(x)=y$. It is clear that both $r$ and $s$ yield splittings of~$f$. The arrows $p_1,p_2$ are the pullback projections and $e_1,e_2$ are the canonical pullback injections. The algebra $P$ is given by all the pairs $(t_1,t_2)\in\mathsf{F}(x,y)\times\mathsf{F}(x,y)$ of binary terms such that the equation $t_1(x,x)=t_2(x,x)$ holds in~$\mathbb{V}$. We note that, since varieties are cocomplete regular categories, the condition that $e_1,e_2$ are jointly extremally epimorphic is equivalent to the condition that their induced morphism $[e_1,e_2]\colon \mathsf{F}(x,y)+\mathsf{F}(x,y)\to P$ is a regular epimorphism, i.e., a surjective homomorphism. We split the proof of Mal'tsev's theorem into two parts (Lemmas~\ref{lem:Mal'tsevif and only ifimage} and~\ref{lem:imageif and only ifp}).

\begin{lemma}\label{lem:Mal'tsevif and only ifimage}
The variety $\mathbb{V}$ is a Mal'tsev category if and only if $(y,x)\in\Image([e_1,e_2])$.
\end{lemma}

\begin{proof}
"$\Rightarrow$": If $\mathbb{V}$ is a Mal'tsev category, the map $[e_1,e_2]$ is a surjection. The claim follows since clearly $(y,x)\in P$.
	
"$\Leftarrow$": Let
\begin{equation}\label{diagr:general pullback split epis}
	\begin{tikzcd}
		{P'} && B \\
		\\
		A && C
		\arrow["{r'}"', shift right=1, from=3-3, to=3-1]
		\arrow["{f'}"', shift right=1, from=3-1, to=3-3]
		\arrow["{g'}", shift left=1, from=1-3, to=3-3]
		\arrow["{s'}", shift left=1, from=3-3, to=1-3]
		\arrow["{p_1'}"', shift right=1, from=1-1, to=3-1]
		\arrow["{e_1'}"', shift right=1, from=3-1, to=1-1]
		\arrow["{p_2'}", shift left=1, from=1-1, to=1-3]
		\arrow["{e_2'}", shift left=1, from=1-3, to=1-1]
		\arrow["\scalebox{2}{$\lrcorner$}"{anchor=center, pos=0.125}, draw=none, from=1-1, to=3-3]
	\end{tikzcd}
\end{equation}
be a pullback of split epimorphisms $f',g'$ in $\mathbb{V}$ with $r',s'$ the respective splittings and $e_1',e_2'$ the canonical inclusions. We show that $[e_1',e_2']\colon A+B\to P'$ is surjective. To this end, let $(a,b)\in P'=\{(a',b')\in A\times B\mid f'(a')=g'(b')\}$ be arbitrary. It is easy to see that the unique morphism $\alpha\colon \mathsf{F}(x,y)\to A$ such that $\alpha(x)=r'f'(a)$ and $\alpha(y)=a$, and the unique morphism $\beta\colon \mathsf{F}(x,y)\to B$ such that $\beta(x)=b$ and $\beta(y)=s'g'(b)$ induce a unique morphism $\delta\colon P\to P'$ such that the cube
\begin{equation*}
	\xymatrix@C=2.5pc{
		P \ar@<-2pt>[dd]_-{p_1} \ar@<2pt>[rr]^-{p_2} \ar@{.>}@<-4pt>[rrrd]^-{\delta} && {\mathsf{F}(x,y)} \ar@<2pt>[ll]^-{e_2} \ar@<2pt>[dd]^(.6){f}|(.39){\hole} \ar@<5pt>[rrrd]^-{\beta} \\
		&&& P' \ar@<-2pt>[dd]_(.4){p'_1} \ar@<2pt>[rr]^(.4){p'_2} && B \ar@<2pt>[ll]^-{e_2'} \ar@<2pt>[dd]^-{g'} \\
		{\mathsf{F}(x,y)} \ar@<-2pt>[rr]_(.6){f} \ar@<-2pt>[uu]_-{e_1} \ar@<-2pt>[rrrd]_-{\alpha} && {\mathsf{F}(x)} \ar@<2pt>[uu]^-{s}|(.63){\hole} \ar@<-2pt>[ll]_-{r} \ar@<5pt>[rrrd]|(.38){\hole}|(.42){\hole}^-{\gamma} \\
		&&& A \ar@<-2pt>[uu]_(.6){e_1'} \ar@<-2pt>[rr]_-{f'} && C \ar@<-2pt>[ll]_(.55){r'} \ar@<2pt>[uu]^-{s'}
	}
\end{equation*}
`reasonably' commutes, where $\gamma\colon \mathsf{F}(x)\to C$ is the unique morphism such that $\gamma(x)=f'(a)=g'(b)$. Hence we get the commutative diagram
\begin{equation}\label{diagr:pullbackofsplitepis}
	\begin{tikzcd}
		{\mathsf{F}(x,y)+\mathsf{F}(x,y)} && P \\
		\\
		{A+B} && {P'.}
		\arrow["{\alpha+\beta}"', from=1-1, to=3-1]
		\arrow["{[e_1,e_2]}", from=1-1, to=1-3]
		\arrow["{\delta}", from=1-3, to=3-3]
		\arrow["{[e_1',e_2']}"', from=3-1, to=3-3]
	\end{tikzcd}
\end{equation}
By assumption, there exists $X\in \mathsf{F}(x,y)+\mathsf{F}(x,y)$ with $[e_1,e_2](X)=(y,x)$. Thus,
\begin{equation*}
	(a,b)=\delta(y,x)=\delta[e_1,e_2](X)=[e_1',e_2'](\alpha+\beta)(X)\in\Image([e_1',e_2']),
\end{equation*}
which proves that $[e_1',e_2']$ is surjective since $(a,b)$ was arbitrary in~$P'$.
\end{proof}

\begin{lemma}\label{lem:imageif and only ifp}
The image of $[e_1,e_2]\colon \mathsf{F}(x,y)+\mathsf{F}(x,y)\to P$ contains $(y,x)$ if and only if there exists a ternary term $p\in\mathsf{F}(x,y,z)$ such that $p(x,x,y)=y$ and $p(x,y,y)=x$.
\end{lemma}

\begin{proof}
We show that
\begin{equation}\label{eq:im([e_1,e_2])}
	\Image([e_1,e_2])=\{(p(x,x,y),p(x,y,y))\mid p\in \mathsf{F}(x,y,z)\},
\end{equation}
which implies the claim.
	
"$\supseteq$": We know that $(x,x)=e_2(x)$, $(x,y)=e_1(x)=e_2(y)$ and $(y,y)=e_1(y)$ are in $\Image([e_1,e_2])$. Hence, given a ternary term $p\in \mathsf{F}(x,y,z)$,
\begin{equation*}
	(p(x,x,y),p(x,y,y)) = p((x,x),(x,y),(y,y))
\end{equation*}
lies also in the subalgebra $\Image([e_1,e_2])\subseteq P$.
	
"$\subseteq$": By Equation~\eqref{eq:coproduct}, we know that we can write an arbitrary element of $\mathsf{F}(x,y)+\mathsf{F}(x,y)$ as
\begin{equation*}
	\tau(\iota_1(s_1),\ldots,\iota_1(s_k),\iota_2(t_1),\ldots,\iota_2(t_{\ell}))
\end{equation*}
where $k,\ell\geqslant 0$ are integers, $s_1,\ldots,s_k ,t_1,\ldots,t_{\ell}\in \mathsf{F}(x,y)$, $\iota_1,\iota_2\colon \mathsf{F}(x,y)\to \mathsf{F}(x,y)+\mathsf{F}(x,y)$ are the coproduct inclusions of $\mathsf{F}(x,y)+\mathsf{F}(x,y)$ and $\tau$ is a $(k+\ell)$-ary term. We compute
{\allowdisplaybreaks
\begin{align*}
	&[e_1,e_2]\Bigl(\tau\bigl(\iota_1(s_1),\ldots,\iota_1(s_k),\iota_2(t_1),\ldots,\iota_2(t_{\ell})\bigr)\Bigr)\\*
	&=\tau\Bigl([e_1,e_2]\bigl(\iota_1(s_1)\bigr),\ldots,[e_1,e_2]\bigl(\iota_1(s_k)\bigr),[e_1,e_2]\bigl(\iota_2(t_1)\bigr),\ldots,[e_1,e_2]\bigl(\iota_2(t_{\ell})\bigr)\Bigr)\\	
	&=\tau\bigl(e_1(s_1),\ldots,e_1(s_k),e_2(t_1),\ldots,e_2(t_{\ell})\bigr)\\
	&=\tau\Bigl(s_1\bigl(e_1(x),e_1(y)\bigr),\ldots,s_k\bigl(e_1(x),e_1(y)\bigr),t_1(e_2(x),e_2(y)),\ldots,t_{\ell}(e_2(x),e_2(y))\Bigr)\\
	&=\tau\Bigl(s_1\bigl((x,y),(y,y)\bigr),\ldots,s_k\bigl((x,y),(y,y)\bigr),t_1\bigl((x,x),(x,y)\bigr),\ldots,t_{\ell}\bigl((x,x),(x,y)\bigr)\Bigr)\\
	&=\tau\Bigl(\bigl(s_1(x,y),s_1(y,y)\bigr),\ldots,\bigl(s_k(x,y),s_k(y,y)\bigr),\bigl(t_1(x,x),t_1(x,y)\bigr),\ldots,\bigl(t_{\ell}(x,x),t_{\ell}(x,y)\bigr)\Bigr)\\
	&=\Bigl(\tau\bigl(s_1(x,y),\ldots,s_k(x,y),t_1(x,x),\ldots,t_{\ell}(x,x)\bigr),\\
	&\hspace{226pt}\tau\bigl(s_1(y,y),\ldots,s_k(y,y),t_1(x,y),\ldots,t_{\ell}(x,y)\bigr)\Bigr)\\*
	&=\bigl(p(x,x,y),p(x,y,y)\bigr),
\end{align*}}%
where we set $p(x,y,z):=\tau(s_1(y,z),\ldots,s_k(y,z),t_1(x,y),\ldots,t_\ell(x,y))$. Hence the claim is proven.
\end{proof}

If the variety $\mathbb{V}$ is a weakly Mal'tsev category, the morphism $[e_1,e_2]\colon \mathsf{F}(x,y)+\mathsf{F}(x,y)\to P$ is only an epimorphism. Let us consider its cokernel pair
\begin{equation*}
	\begin{tikzcd}
		{\mathsf{F}(x,y)+\mathsf{F}(x,y)} && P && Q.
		\arrow["{[e_1,e_2]}", from=1-1, to=1-3]
		\arrow["{q_1}", shift left=1, from=1-3, to=1-5]
		\arrow["{q_2}"', shift right=1, from=1-3, to=1-5]
	\end{tikzcd}
\end{equation*}

\begin{lemma}\label{lem:WMif and only ifprojection}
The variety $\mathbb{V}$ is a weakly Mal'tsev category if and only if $q_1(y,x)=q_2(y,x)$. 
\end{lemma}

\begin{proof}
"$\Rightarrow$": If $\mathbb{V}$ is a weakly Mal'tsev category, the arrow $[e_1,e_2]$ is an epimorphism. Hence $q_1=q_2$.
	
"$\Leftarrow$": As in the proof of Lemma~\ref{lem:Mal'tsevif and only ifimage}, we consider a general pullback of split epimorphisms in~$\mathbb{V}$ as in Diagram~\eqref{diagr:general pullback split epis}. We show that $[e_1',e_2']$ is an epimorphism. Let $(a,b)\in P'$ be arbitrary. We consider the extended version
\begin{equation*}
	\begin{tikzcd}
		{\mathsf{F}(x,y)+\mathsf{F}(x,y)} && P && Q \\
		\\
		{A+B} && {P'} && {Q'}
		\arrow["{\alpha+\beta}"', from=1-1, to=3-1]
		\arrow["{[e_1,e_2]}", from=1-1, to=1-3]
		\arrow["{\delta}", from=1-3, to=3-3]
		\arrow["{[e_1',e_2']}"', from=3-1, to=3-3]
		\arrow["{q_1}", shift left=1, from=1-3, to=1-5]
		\arrow["{q_1'}", shift left=1, from=3-3, to=3-5]
		\arrow["{\rho}", dotted, from=1-5, to=3-5]
		\arrow["{q_2}"', shift right=1, from=1-3, to=1-5]
		\arrow["{q_2'}"', shift right=1, from=3-3, to=3-5]
	\end{tikzcd}
\end{equation*}
of Diagram~\eqref{diagr:pullbackofsplitepis}, where $(Q',q_1',q_2')$ is the cokernel pair of $[e_1',e_2']$ and $\rho\colon Q\to Q'$ is the unique morphism such that $\rho q_i=q'_i \delta$ for each $i\in\{1,2\}$. We have that
\begin{equation*}
	\begin{aligned}
		q_1'(a,b)
		=q_1'\delta(y,x)
		=\rho q_1(y,x)
		=\rho q_2(y,x)
		=q_2'\delta(y,x)
		=q_2'(a,b).
	\end{aligned}
\end{equation*}
Hence $q_1'=q_2'$ and the claim is proven.
\end{proof}

We are now ready to state and prove our main theorem.

\begin{theorem}\label{thm:syntaxWM}
A finitary one-sorted variety $\mathbb{V}$ of universal algebras is a weakly Mal'tsev category if and only if there exist integers $k,m,N\geqslant 0$, binary terms $f_1,g_1,\ldots,f_k,g_k\in \mathsf{F}(x,y)$, ternary terms $p_1,\ldots,p_m\in\mathsf{F}(x,y,z)$, $(2(k+2m+1))$-ary terms $s_1,\ldots,s_N$, ${(2(k+m+2))}$-ary terms $\sigma_1,\ldots,\sigma_{N+1}$ and, for all $i\in\{1,\ldots,N+1\}$, $(k+m+1)$-ary terms $\eta^{(i)}_1,\eta^{(i)}_2,\epsilon^{(i)}_1,\epsilon^{(i)}_2$ such that the following identities (on variables $x$, $y$, $u$, $u'$, $v_1,\dots,v_k$, $v'_1,\dots,v'_k$, $w_1,\dots,w_m$ and $w'_1,\dots,w'_m$) hold, where we write $\vec{v}$ for $v_1,\ldots,v_k$ and $\vec{w}$ for $w_1,\ldots,w_m$, and analogously for $\vec{v}'$ and~$\vec{w}'$:
\begin{align}\label{eq:(f_i,g_i)inP}
	f_i(x,x)=g_i(x,x)
\end{align} 
for all $i\in\{1,\ldots,k\}$;
{\allowdisplaybreaks
\begin{subequations}\label{subeq:WMeta}
	\begin{alignat}{2}
		&\eta^{(i)}_\alpha(y,f_1(x,y),\ldots,f_k(x,y),p_1(x,x,y),\ldots,p_m(x,x,y))\nonumber \\*
		=&\epsilon^{(i)}_\alpha(y,f_1(x,y),\ldots,f_k(x,y),p_1(x,x,y),\ldots,p_m(x,x,y)),\label{eq:WMetay}\\[3ex]
		&\eta^{(i)}_\alpha(x,g_1(x,y),\ldots,g_k(x,y),p_1(x,y,y),\ldots,p_m(x,y,y))\nonumber \\*
		=&\epsilon^{(i)}_\alpha(x,g_1(x,y),\ldots,g_k(x,y),p_1(x,y,y),\ldots,p_m(x,y,y)),\label{eq:WMetax}
	\end{alignat}
\end{subequations}}%
for all $i\in\{1,\ldots,N+1\}$ and $\alpha\in\{1,2\}$;
{\allowdisplaybreaks
\begin{subequations}
	\begin{alignat}{2}
		&\sigma_{i}(u,\vec{v},\vec{w},u',\vec{v}',\vec{w}',\epsilon^{(i)}_1(u,\vec{v},\vec{w}),\epsilon^{(i)}_2(u',\vec{v}',\vec{w}'))\nonumber\\*
		=&s_{i}(u,\vec{v},\vec{w},\vec{w},u',\vec{v}',\vec{w}',\vec{w}'),\label{eq:WModd}\\[3ex]
		&s_{i}(u,\vec{v},\vec{w},\vec{w}',u',\vec{v}',\vec{w}',\vec{w})\nonumber \\*
		=&\sigma_{i+1}(u,\vec{v},\vec{w},u',\vec{v}',\vec{w}',{\eta}^{(i+1)}_1(u,\vec{v},\vec{w}),{\eta}^{(i+1)}_2(u',\vec{v}',\vec{w}')),\label{eq:WMeven}
	\end{alignat}
\end{subequations}}%
for all $i\in\{1,\ldots,N\}$ and
\begin{subequations}\label{subeq:WMuu'}
	\begin{alignat}{2}
		u
		&=\sigma_1(u,\vec{v},\vec{w},u',\vec{v}',\vec{w}',\eta^{(1)}_1(u,\vec{v},\vec{w}),\eta^{(1)}_2(u',\vec{v}',\vec{w}')),\label{eq:WMu}\\
		u'
		&=\sigma_{N+1}(u,\vec{v},\vec{w},u',\vec{v}',\vec{w}',\epsilon^{(N+1)}_1(u,\vec{v},\vec{w}),\epsilon^{(N+1)}_2(u',\vec{v}',\vec{w}'))\label{eq:WMu'}.
	\end{alignat}
\end{subequations}
\end{theorem}

Note that, in comparison to Equation~\eqref{eq:WModd}, the second occurrences of $\vec{w}$ and $\vec{w}'$ in the term $s_i$ are swapped in Equation~\eqref{eq:WMeven}.

\begin{proof}
We start by showing that if $\mathbb{V}$ is a weakly Mal'tsev category, then the terms mentioned in the statement exist. We first treat the case where $\mathbb{V}$ is a Mal'tsev category separately. If a Mal'tsev term $p\in\mathsf{F}(x,y,z)$, which fulfills $p(x,x,y)=y$ and $p(x,y,y)=x$, exists, we can easily construct terms that satisfy the requirements of the theorem. We can choose for instance $k$ to be $0$ and $m,N$ to be~$1$, and we can set
\begin{align*}
	p_1(x,y,z)&:= p(x,y,z),\\
	s_1(u,w,\tilde{w},u',w',\tilde{w}')&:= \tilde{w},\\
	\sigma_1(u,w,u',w',a,b)&:= a,\\
	\sigma_2(u,w,u',w',a,b)&:= b,\\
	\eta^{(1)}_1(u,w):= \epsilon^{(2)}_2(u,w)&:= u,\\
	\epsilon^{(1)}_1(u,w):= \eta^{(2)}_2(u,w)&:= w. 
\end{align*}
Then the binary terms $\eta^{(1)}_2,\epsilon^{(1)}_2$, $\eta^{(2)}_1,\epsilon^{(2)}_1$ have only to satisfy the Equations~\eqref{subeq:WMeta}, which can easily be attained.

Let us now suppose that $\mathbb{V}$ is a weakly Mal'tsev variety but not a Mal'tsev variety. By Lemma~\ref{lem:WMif and only ifprojection}, we know that $\mathbb{V}$ is weakly Mal'tsev if and only if $q_1(y,x)=q_2(y,x)$. We construct the cokernel pair $(q_1,q_2\colon P\to Q)$ of the map $[e_1,e_2]\colon \mathsf{F}(x,y)+\mathsf{F}(x,y)\to P$ by means of the coequalizer $q$ of the maps $\iota_1[e_1,e_2],\iota_2[e_1,e_2]$, where $\iota_1,\iota_2$ are the coproduct inclusions into $P+P$, as in the following diagram:
\begin{equation*}
	\xymatrix{\mathsf{F}(x,y)+\mathsf{F}(x,y) \ar[rr]^-{[e_1,e_2]} \ar[dd]_-{[e_1,e_2]} && P \ar[dd]^-{\iota_2}  \ar@/^1.2pc/[rddd]^-{q_2} \\
		\\
		P \ar[rr]_-{\iota_1} \ar@/_1.2pc/[rrrd]_{q_1} && {P+P} \ar[rd]^-{q} \\
		&&& Q}
\end{equation*}
The algebra $Q$ is given by the quotient of $P+P$ with respect to, due to Equation~\eqref{eq:im([e_1,e_2])}, the smallest congruence $C\subseteq(P+P)^2$ containing the subset
\begin{equation*}
	S:= \Bigl\{\Bigl(\iota_1\bigl(p(x,x,y),p(x,y,y)\bigr),\iota_2\bigl(p(x,x,y),p(x,y,y)\bigr)\Bigr)\in (P+P)^2\mid p\in\mathsf{F}(x,y,z)\Bigr\}.
\end{equation*}
More precisely, we first define the subset
\begin{align*}
	S_{\mathrm{S}}:= &\Bigl\{\Bigl(\iota_1\bigl(p(x,x,y),p(x,y,y)\bigr),\iota_2\bigl(p(x,x,y),p(x,y,y)\bigl)\Bigr)\in (P+P)^2\mid p\in\mathsf{F}(x,y,z)\Bigr\}\\
	 \cup &\Bigl\{\Bigl(\iota_2\bigl(p(x,x,y),p(x,y,y)\bigr),\iota_1\bigl(p(x,x,y),p(x,y,y)\bigr)\Bigr)\in (P+P)^2\mid p\in\mathsf{F}(x,y,z)\Bigr\},
\end{align*}
which yields the smallest symmetric relation on $P+P$ containing~$S$. Secondly, we define
\begin{equation*}
	S_{\mathrm{SR}}:= S_{\mathrm{S}}\cup \{(X,X)\in(P+P)^2\mid X\in P+P \},
\end{equation*}
which yields the smallest reflexive symmetric relation on $P+P$ containing~$S$. Thirdly, we define
\begin{align*}
	S_{\mathrm{SRO}}:= \{t((X_1,Y_1),\ldots,(X_n,Y_n))\in (P+P)^2 \mid\, & n\geqslant 0 \text{ is an integer,}\\
	&(X_1,Y_1),\ldots,(X_n,Y_n)\in S_{\mathrm{SR}}\\
	&\text{and $t$ is an $n$-ary term}\},
\end{align*}
yielding the smallest symmetric reflexive relation on $P+P$ containing $S$ and being closed under operations. Fourthly, we define
\begin{align*}
	S_{\mathrm{SROT}}:= \{(X,Y)\in (P+P)^2\mid&\text{ there exists an integer $n\geqslant 1$ and $X_1,\ldots,X_{n+1}\in P+P$}\\ 
			&\text{ such that } (X_1,X_2),(X_2,X_3),\ldots,(X_n,X_{n+1})\in S_{\mathrm{SRO}}\\
			& \text{ and $X=X_1$ and $X_{n+1}=Y$}\}
\end{align*}
to be the transitive closure of~$S_{\mathrm{SRO}}$. It is well-known and easy to prove that $C$ coincides with~$S_{\mathrm{SROT}}$. Thus, the condition $q_1(y,x)=q_2(y,x)$ is equivalent to $\bigl(\iota_1(y,x),\iota_2(y,x)\bigr)\in S_{\mathrm{SROT}}$. This means that there exist an integer $n\geqslant 1$ and elements $X_1,\ldots,X_{n+1}\in P+P$ such that
\begin{equation*}
	(X_1,X_2),(X_2,X_3),\ldots,(X_{n-1},X_n),(X_n,X_{n+1})\in S_{\mathrm{SRO}},
\end{equation*}
$\iota_1(y,x)=X_1$ and $X_{n+1}=\iota_2(y,x)$. For $i\in\{1,\ldots,n\}$, the condition $(X_i,X_{i+1})\in S_{\mathrm{SRO}}$ means that there exist an integer $a_i\geqslant 0$, elements $(X^{(i)}_1,Y^{(i)}_1),\ldots,(X^{(i)}_{a_i},Y^{(i)}_{a_i})\in S_{\mathrm{SR}}$ and an $a_i$-ary term $t_i$ such that
\begin{align*}
	(X_i,X_{i+1})=(t_i(X^{(i)}_1,\ldots,X^{(i)}_{a_i}),t_i(Y^{(i)}_1,\ldots,Y^{(i)}_{a_i})).
\end{align*}
By definition of~$S_{\mathrm{SR}}$, each pair $(X^{(i)}_j,Y^{(i)}_j)$ is either of the form
\begin{equation*}
	\Bigl(\iota_1\bigl(p^{(i)}_j(x,x,y),p^{(i)}_j(x,y,y)\bigr),\iota_2\bigl(p^{(i)}_j(x,x,y),p^{(i)}_j(x,y,y)\bigr)\Bigr),
\end{equation*}
where $p^{(i)}_j\in \mathsf{F}(x,y,z)$, of the form
\begin{equation*}
	\Bigl(\iota_2\bigl(q^{(i)}_j(x,x,y),q^{(i)}_j(x,y,y)\bigr),\iota_1\bigl(q^{(i)}_j(x,x,y),q^{(i)}_j(x,y,y)\bigl)\Bigr),
\end{equation*}
where $q^{(i)}_j\in \mathsf{F}(x,y,z)$, or of the form
\begin{equation*}
	(Z^{(i)}_{j},Z^{(i)}_j),
\end{equation*}
where $Z^{(i)}_j=X^{(i)}_j=Y^{(i)}_j\in P+P$. Hence we see that, by possibly permuting the variables in the term~$t_i$, the condition $(X_i,X_{i+1})\in S_{\mathrm{SRO}}$ is equivalent to the existence of integers $m_i, m'_i, m''_i\geqslant 0$, ternary terms $p^{(i)}_1,\ldots,p^{(i)}_{m_i}$ and $q^{(i)}_1,\ldots,q^{(i)}_{m'_i}$, elements $Z^{(i)}_1,\ldots,Z^{(i)}_{m''_i}\in P+P$ and an $(m_i+m'_i+m''_i)$-ary term $t_i$ such that the identities
\begin{align*}
	X_i = t_i\Bigl(&Z^{(i)}_1,\ldots,Z^{(i)}_{m''_i},\\
	&\iota_1\bigl(p^{(i)}_1(x,x,y),p^{(i)}_1(x,y,y)\bigr),\ldots,\iota_1\bigl(p^{(i)}_{m_i}(x,x,y),p^{(i)}_{m_i}(x,y,y)\bigr),\\
	&\iota_2\bigl(q^{(i)}_1(x,x,y),q^{(i)}_1(x,y,y)\bigr),\ldots,\iota_2\bigl(q^{(i)}_{m'_i}(x,x,y),q^{(i)}_{m'_i}(x,y,y)\bigr) \Bigr)
\end{align*}
and
\begin{align*}
	X_{i+1} = t_i\Bigl(&Z^{(i)}_1,\ldots,Z^{(i)}_{m''_i},\\
	&\iota_2\bigl(p^{(i)}_1(x,x,y),p^{(i)}_1(x,y,y)\bigr),\ldots,\iota_2\bigl(p^{(i)}_{m_i}(x,x,y),p^{(i)}_{m_i}(x,y,y)\bigr),\\
	&\iota_1\bigl(q^{(i)}_1(x,x,y),q^{(i)}_1(x,y,y)\bigr),\ldots,\iota_1\bigl(q^{(i)}_{m'_i}(x,x,y),q^{(i)}_{m'_i}(x,y,y)\bigr) \Bigr)
\end{align*}
hold in $P+P$. Given a term $t\in\mathsf{F}(x_1,\ldots,x_v)$, where $v\geqslant 0$ is an integer, we can add a variable $x_{v+1}$ to $t$ in the sense that we define a term $t'\in\mathsf{F}(x_1,\ldots,x_{v+1})$ by setting $t'(x_1,\ldots,x_{v+1}):= t(x_1,\ldots,x_v)$. In this spirit, we can consider the unions
\begin{equation*}
	\bigcup_{i\in\{1,\ldots,n\}}\{p^{(i)}_j\mid j\in\{1,\ldots,m_i\}\},
\end{equation*}
\begin{equation*}
	\bigcup_{i\in\{1,\ldots,n\}}\{q^{(i)}_j\mid j\in\{1,\ldots,m'_i\}\}
\end{equation*}
and
\begin{equation*}
	\bigcup_{i\in\{1,\ldots,n\}}\{Z^{(i)}_j\mid j\in\{1,\ldots,m''_i\}\},
\end{equation*}
and assume, without loss of generality, that $m_1=\cdots =m_n=: m$, $m'_1=\cdots =m'_n=: m'$ and $m''_1=\cdots =m''_n=: m''$. Hence we can omit the upper index from the elements $p^{(i)}_j,q^{(i)}_j,Z^{(i)}_j$. Furthermore, by looking at the union
\begin{equation*}
	\{p_j\mid j\in\{1,\ldots,m\}\} \cup \{q_j\mid j\in \{1,\ldots,m'\}\},
\end{equation*}
we can assume, without loss of generality, that $m=m'$ and $p_j=q_j$ for all $j\in\{1,\ldots,m\}$. Thus, the condition that $(X_i,X_{i+1})\in S_{\mathrm{SRO}}$ for all $i\in\{1,\ldots,n\}$ is equivalent to the existence of integers $m,m''\geqslant 0$, terms $p_1,\ldots,p_m\in\mathsf{F}(x,y,z)$, elements $Z_1,\ldots,Z_{m''}\in P+P$ and $(2m+m'')$-ary terms $t_1,\ldots,t_n$ such that the identities
\begin{align}\label{eq:Xi=ti}
	X_i = t_i \Bigl(&Z_1,\ldots,Z_{m''},\\
	& \iota_1\bigl(p_1(x,x,y),p_1(x,y,y)\bigr),\ldots,\iota_1\bigl(p_m(x,x,y),p_m(x,y,y)\bigr)\nonumber\\
	& \iota_2\bigl(p_1(x,x,y),p_1(x,y,y)\bigr),\ldots,\iota_2\bigl(p_m(x,x,y),p_m(x,y,y)\bigr)\Bigr)\nonumber
\end{align}
and
\begin{align}\label{eq:Xi+1=ti}
	X_{i+1} = t_i \Bigl(&Z_1,\ldots,Z_{m''},\\
	& \iota_2\bigl(p_1(x,x,y),p_1(x,y,y)\bigr),\ldots,\iota_2\bigl(p_m(x,x,y),p_m(x,y,y)\bigr)\nonumber\\
	& \iota_1\bigl(p_1(x,x,y),p_1(x,y,y)\bigr),\ldots,\iota_1\bigl(p_m(x,x,y),p_m(x,y,y)\bigr)\Bigr)\nonumber
\end{align}
hold in $P+P$ for all $i\in\{1,\ldots,n\}$. By Equation~\eqref{eq:coproduct}, there exist, for all $i\in\{1,\ldots,m''\}$, integers $\hat{k}_i,\tilde{k}_i\geqslant 0$, elements $(\hat{f}^{(i)}_1,\hat{g}^{(i)}_1),\ldots,(\hat{f}^{(i)}_{\hat{k}_i},\hat{g}^{(i)}_{\hat{k}_i})$ and $(\tilde{f}^{({i})}_1,\tilde{g}^{(i)}_1),\ldots,(\tilde{f}^{({i})}_{\tilde{k}_i},\tilde{g}^{(i)}_{\tilde{k}_i})$ in~$P$, and a $(\hat{k}_i+\tilde{k}_i)$-ary term $s_i$ such that
\begin{equation*}
	Z_i=s_i\bigl(\iota_1(\hat{f}^{(i)}_1,\hat{g}^{(i)}_1),\ldots,\iota_1(\hat{f}^{(i)}_{\hat{k}_i},\hat{g}^{(i)}_{\hat{k}_i}),\iota_2(\tilde{f}^{(i)}_1,\tilde{g}^{(i)}_1),\ldots,\iota_2(\tilde{f}^{(i)}_{\tilde{k}_i},\tilde{g}^{(i)}_{\tilde{k}_i})\bigr).
\end{equation*}
Proceeding analogously as we did for the $p^{(i)}_1,\ldots,p^{(i)}_{m_i}$ and $q^{(i)}_1,\ldots,q^{(i)}_{m'_i}$ from above, we can first assume, without loss of generality, that $\hat{k}_1=\cdots=\hat{k}_{m''}=: \hat{k}$ and $\tilde{k}_1=\cdots=\tilde{k}_{m''}=: \tilde{k}$ and omit the upper index from the elements $\hat{f}^{(i)}_j,\hat{g}^{(i)}_j,\tilde{f}^{(i)}_j,\tilde{g}^{(i)}_j$, and then that $\hat{k}=\tilde{k}=:k$, and $\hat{f}_j=\tilde{f}_j$ and $\hat{g}_j=\tilde{g}_j$ for all $j\in\{1,\ldots,k\}$. In a next step, we see that we can assume without loss of generality that we have an integer $k\geqslant 0$, binary terms $f_1,g_1,\ldots,f_k,g_k$ such that $(f_j,g_j)\in P$, i.e., $f_j(x,x)=g_j(x,x)$, for all $j\in\{1,\ldots,k\}$, and $(2(k+m+1))$-ary terms $s_1,\ldots,s_{m''}$ such that the $k+m+1$ elements
$$(y,x),(f_1,g_1)\dots,(f_k,g_k),(p_1(x,x,y),p_1(x,y,y))\dots,(p_m(x,x,y),p_m(x,y,y))$$
of $P$ are pairwise distinct and 
\begin{align*}
	Z_i=s_i(&\iota_1(y,x),\iota_1(f_1,g_1),\ldots,\iota_1(f_k,g_k),\\
	&\iota_1(p_1(x,x,y),p_1(x,y,y)),\ldots,\iota_1(p_m(x,x,y),p_m(x,y,y)),\\
	&\iota_2(y,x),\iota_2(f_1,g_1),\ldots,\iota_2(f_k,g_k),\\
	&\iota_2(p_1(x,x,y),p_1(x,y,y)),\ldots,\iota_2(p_m(x,x,y),p_m(x,y,y)))
\end{align*}
for all $i\in\{1,\ldots,m''\}$. Notice that $(y,x)$ cannot be of the form $(p_j(x,x,y),p_j(x,y,y))$ since we supposed that $\mathbb{V}$ is not a Mal'tsev variety. By setting, for all $i\in\{1,\ldots,n\}$, the $(2(k+2m+1))$-ary term
\begin{align*}
	&t'_i(u,v_1,\ldots,v_k,w_1,\ldots,w_m,\widetilde{w}_1,\dots,\widetilde{w}_m,u',v'_1,\ldots,v'_k,w'_1,\ldots,w'_m,\widetilde{w}'_1,\dots,\widetilde{w}'_m)\\
	:=\, &t_i
	\begin{aligned}[t]
	(& s_1(u,v_1,\ldots,v_k,w_1,\ldots,w_m,u',v'_1,\ldots,v'_k,w'_1,\ldots,w'_m)\ldots,\\
	&s_{m''}(u,v_1,\ldots,v_k,w_1,\ldots,w_m,u',v'_1,\ldots,v'_k,w'_1,\ldots,w'_m),\widetilde{w}_1,\dots,\widetilde{w}_m,\widetilde{w}'_1,\dots,\widetilde{w}'_m),
	\end{aligned}
\end{align*} 
Equations~\eqref{eq:Xi=ti} and~\eqref{eq:Xi+1=ti} translate into
\begin{align*}
	X_i &= t'_i
	\begin{aligned}[t]
		\Bigl(&\iota_1(y,x),\iota_1(f_1,g_1),\ldots,\iota_1(f_k,g_k),\\
		&\iota_1\bigl(p_1(x,x,y),p_1(x,y,y)\bigr),\ldots,\iota_1\bigl(p_m(x,x,y),p_m(x,y,y)\bigr),\\
		&\iota_1\bigl(p_1(x,x,y),p_1(x,y,y)\bigr),\ldots,\iota_1\bigl(p_m(x,x,y),p_m(x,y,y)\bigr),\\
		&\iota_2(y,x),\iota_2(f_1,g_1),\ldots,\iota_2(f_k,g_k),\\
		&\iota_2\bigl(p_1(x,x,y),p_1(x,y,y)\bigr),\ldots,\iota_2\bigl(p_m(x,x,y),p_m(x,y,y)\bigr),\\
		&\iota_2\bigl(p_1(x,x,y),p_1(x,y,y)\bigr),\ldots,\iota_2\bigl(p_m(x,x,y),p_m(x,y,y)\bigr)\Bigr)
	\end{aligned}\\
	&=: t'_i\{1,2\}
\end{align*}
and
\begin{align*}
	X_{i+1} &= t'_i
	\begin{aligned}[t]
		\Bigl(&\iota_1(y,x),\iota_1(f_1,g_1),\ldots,\iota_1(f_k,g_k),\\
		&\iota_1\bigl(p_1(x,x,y),p_1(x,y,y)\bigr),\ldots,\iota_1\bigl(p_m(x,x,y),p_m(x,y,y)\bigr),\\
		&\iota_2\bigl(p_1(x,x,y),p_1(x,y,y)\bigr),\ldots,\iota_2\bigl(p_m(x,x,y),p_m(x,y,y)\bigr),\\
		&\iota_2(y,x),\iota_2(f_1,g_1),\ldots,\iota_2(f_k,g_k),\\
		&\iota_2\bigl(p_1(x,x,y),p_1(x,y,y)\bigr),\ldots,\iota_2\bigl(p_m(x,x,y),p_m(x,y,y)\bigr),\\
		&\iota_1\bigl(p_1(x,x,y),p_1(x,y,y)\bigr),\ldots,\iota_1\bigl(p_m(x,x,y),p_m(x,y,y)\bigr)\Bigr)
	\end{aligned}\\
	&=: t'_i\{2,1\}
\end{align*}
respectively. In the notation $t'_i\{1,2\}$ and $t'_i\{2,1\}$, the first (respectively the second) entry inside the brackets corresponds to which of $\iota_1$ or $\iota_2$ is used in the third (respectively the sixth) row in the above expressions. Hence we see that the condition $q_1(y,x)=q_2(y,x)$ implies the existence of integers $k,m\geqslant 0$ and $n\geqslant 1$, binary terms $f_1,g_1,\ldots,f_k,g_k$ with $f_i(x,x)=g_i(x,x)$ for all $i\in\{1,\ldots,k\}$, ternary terms $p_1,\ldots,p_m$ and $(2(k+2m+1))$-ary terms $t'_1,\ldots,t'_n$ such that the equations
\begin{align*}
		\iota_1&(y,x)\\
		=t'_1&
		\begin{aligned}[t]
			\Bigl(&\iota_1(y,x),\iota_1(f_1,g_1),\ldots,\iota_1(f_k,g_k),\\
			&\iota_1\bigl(p_1(x,x,y),p_1(x,y,y)\bigr),\ldots,\iota_1\bigl(p_m(x,x,y),p_m(x,y,y)\bigr),\\
			&\iota_1\bigl(p_1(x,x,y),p_1(x,y,y)\bigr),\ldots,\iota_1\bigl(p_m(x,x,y),p_m(x,y,y)\bigr),\\
			&\iota_2(y,x),\iota_2(f_1,g_1),\ldots,\iota_2(f_k,g_k),\\
			&\iota_2\bigl(p_1(x,x,y),p_1(x,y,y)\bigr),\ldots,\iota_2\bigl(p_m(x,x,y),p_m(x,y,y)\bigr),\\
			&\iota_2\bigl(p_1(x,x,y),p_1(x,y,y)\bigr),\ldots,\iota_2\bigl(p_m(x,x,y),p_m(x,y,y)\bigr)\Bigr)
		\end{aligned}\tag{a}\label{eq:a}
\end{align*}
and, for all $i\in\{1,\dots,n-1\}$,
\begin{align*}
		t'_i&
		\begin{aligned}[t] 
			\Bigl(&\iota_1(y,x),\iota_1(f_1,g_1),\ldots,\iota_1(f_k,g_k),\\
			&\iota_1\bigl(p_1(x,x,y),p_1(x,y,y)\bigr),\ldots,\iota_1\bigl(p_m(x,x,y),p_m(x,y,y)\bigr),\\
			&\iota_2\bigl(p_1(x,x,y),p_1(x,y,y)\bigr),\ldots,\iota_2\bigl(p_m(x,x,y),p_m(x,y,y)\bigr),\\
			&\iota_2(y,x),\iota_2(f_1,g_1),\ldots,\iota_2(f_k,g_k),\\
			&\iota_2\bigl(p_1(x,x,y),p_1(x,y,y)\bigr),\ldots,\iota_2\bigl(p_m(x,x,y),p_m(x,y,y)\bigr),\\
			&\iota_1\bigl(p_1(x,x,y),p_1(x,y,y)\bigr),\ldots,\iota_1\bigl(p_m(x,x,y),p_m(x,y,y)\bigr)\Bigr)
		\end{aligned}\\
		=t'_{i+1}&
		\begin{aligned}[t]
			\Bigl(&\iota_1(y,x),\iota_1(f_1,g_1),\ldots,\iota_1(f_k,g_k),\\
			&\iota_1\bigl(p_1(x,x,y),p_1(x,y,y)\bigr),\ldots,\iota_1\bigl(p_m(x,x,y),p_m(x,y,y)\bigr),\\
			&\iota_1\bigl(p_1(x,x,y),p_1(x,y,y)\bigr),\ldots,\iota_1\bigl(p_m(x,x,y),p_m(x,y,y)\bigr),\\
			&\iota_2(y,x),\iota_2(f_1,g_1),\ldots,\iota_2(f_k,g_k),\\
			&\iota_2\bigl(p_1(x,x,y),p_1(x,y,y)\bigr),\ldots,\iota_2\bigl(p_m(x,x,y),p_m(x,y,y)\bigr),\\
			&\iota_2\bigl(p_1(x,x,y),p_1(x,y,y)\bigr),\ldots,\iota_2\bigl(p_m(x,x,y),p_m(x,y,y)\bigr)\Bigr)
		\end{aligned}\tag{b($\i$)}\label{eq:bi}
\end{align*}
and
\begin{align*}
	t'_n&
	\begin{aligned}[t]
		\Bigl(&\iota_1(y,x),\iota_1(f_1,g_1),\ldots,\iota_1(f_k,g_k),\\
		&\iota_1\bigl(p_1(x,x,y),p_1(x,y,y)\bigr),\ldots,\iota_1\bigl(p_m(x,x,y),p_m(x,y,y)\bigr),\\
		&\iota_2\bigl(p_1(x,x,y),p_1(x,y,y)\bigr),\ldots,\iota_2\bigl(p_m(x,x,y),p_m(x,y,y)\bigr),\\
		&\iota_2(y,x),\iota_2(f_1,g_1),\ldots,\iota_2(f_k,g_k),\\
		&\iota_2\bigl(p_1(x,x,y),p_1(x,y,y)\bigr),\ldots,\iota_2\bigl(p_m(x,x,y),p_m(x,y,y)\bigr),\\
		&\iota_1\bigl(p_1(x,x,y),p_1(x,y,y)\bigr),\ldots,\iota_1\bigl(p_m(x,x,y),p_m(x,y,y)\bigr)\Bigr)
	\end{aligned}\\
=\iota_2&({y},{x})\tag{c}\label{eq:c}
\end{align*}
hold in $P+P$. We display the Equations~\eqref{eq:a}, \eqref{eq:bi} for all $i\in\{1,\ldots,n-1\}$ and \eqref{eq:c} in the diagram
\begin{footnotesize}
\begin{equation}\label{diagr:smalltrail}
	\xymatrixcolsep{4.3mm}\xymatrix{
		{\iota_1(y,x)}		\ar@{=}[r]^{\eqref{eq:a}}
		&{t'_1\{1,2\}} 		\ar@{.}[r]
		&{t'_1\{2,1\}}  	\ar@{=}[r]^-{\eqrefi{eq:bi}{1}}	&{t'_2\{1,2\}} \ar@{.}[r]
		&\cdots				\ar@{.}[r] 
		&{t'_{n-1}\{2,1\}}  	\ar@{=}[rr]^-{\eqrefi{eq:bi}{n-1}}
		&&{t'_n\{1,2\}}  	\ar@{.}[r]	
		&{t'_n\{2,1\}}  	\ar@{=}[r]^{\eqref{eq:c}}
		&{\iota_2(y,x).}
	}
\end{equation}
\end{footnotesize}
By Proposition~\ref{prop:coproduct}, Equation~\eqref{eq:a} means that there exist an integer $\ell_0\geqslant 0$ and elements $c_{(0,1)},\ldots$, $c_{(0,\ell_0)}\in \mathsf{F}(\mathsf{U}P+\mathsf{U}P)$ such that
{\allowdisplaybreaks
\begin{equation*}
	\begin{split}
		(y,x)^1&\sim c_{(0,1)},\\*
		c_{(0,1)}&\sim c_{(0,2)},\ldots,c_{(0,\ell_0-1)}\sim c_{(0,\ell_0)},\\
		c_{(0,\ell_0)}&\begin{aligned}[t]
			\sim t'_1\bigl(
			&(y,x)^1,(f_1,g_1)^1,\ldots,(f_k,g_k)^1,\\*
			&({p_1(x,x,y)},{p_1(x,y,y)})^1,\ldots,({p_m(x,x,y)},{p_m(x,y,y)})^1,\\*
			&({p_1(x,x,y)},{p_1(x,y,y)})^1,\ldots,({p_m(x,x,y)},{p_m(x,y,y)})^1,\\*
			&(y,x)^2,(f_1,g_1)^2,\ldots,(f_k,g_k)^2,\\*
			&({p_1(x,x,y)},{p_1(x,y,y)})^2,\ldots,({p_m(x,x,y)},{p_m(x,y,y)})^2,\\*
			&({p_1(x,x,y)},{p_1(x,y,y)})^2,\ldots,({p_m(x,x,y)},{p_m(x,y,y)})^2\bigr)
		\end{aligned}
	\end{split}
\end{equation*}}%
where we write, for an element $X\in P$, $X^1$ and $X^2$ to indicate that we view them as variables in $\mathsf{F}(\mathsf{U}P+\mathsf{U}P)$ coming from the first and second copy of $\mathsf{U}P$ in $\mathsf{U}P+\mathsf{U}P$ respectively, and $\sim$ is the relation on $\mathsf{F}(\mathsf{U}P+\mathsf{U}P)$ from Proposition~\ref{prop:coproduct}. More explicitly, this means that there exist, for all $j\in\{0,\ldots,\ell_0\}$, integers $a_{(0,j)}, b_{(0,j)}\geqslant 0$, elements $D^{(0,j)}_1,\ldots,D^{(0,j)}_{a_{(0,j)}},E^{(0,j)}_1,\ldots,E^{(0,j)}_{b_{(0,j)}}\in P$, $a_{(0,j)}$-ary terms $\mu^{(0,j)}_1,\lambda^{(0,j)}_1$, $b_{(0,j)}$-ary terms $\mu^{(0,j)}_2,\lambda^{(0,j)}_2$ and $(a_{(0,j)}+b_{(0,j)}+2)$-ary terms $\tau_{(0,j)}$ such that the equations
\begin{align*}
	&(y,x)^1\\
	=&\tau_{(0,0)}
	\begin{aligned}[t]
		\Bigl(&(D^{(0,0)}_1)^1,\ldots,(D^{(0,0)}_{a_{(0,0)}})^1,
		(E^{(0,0)}_1)^2,\ldots,(E^{(0,0)}_{b_{(0,0)}})^2,\\
		&\mu^{(0,0)}_1\bigl((D^{(0,0)}_1)^1,\ldots,(D^{(0,0)}_{a_{(0,0)}})^1\bigr),
		\mu^{(0,0)}_2\bigl((E^{(0,0)}_1)^2,\ldots,(E^{(0,0)}_{b_{(0,0)}})^2\bigr)\Bigr)
	\end{aligned}\tag{B1}\label{eq:B1}
\end{align*}
and
\begin{align*}
	&\tau_{(0,0)}
	\begin{aligned}[t]
		\Bigl(&(D^{(0,0)}_1)^1,\ldots,(D^{(0,0)}_{a_{(0,0)}})^1,
		(E^{(0,0)}_1)^2,\ldots,(E^{(0,0)}_{b_{(0,0)}})^2,\\
		&\lambda^{(0,0)}_1\bigl((D^{(0,0)}_1)^1,\ldots,(D^{(0,0)}_{a_{(0,0)}})^1\bigr),
		\lambda^{(0,0)}_2\bigl((E^{(0,0)}_1)^2,\ldots,(E^{(0,0)}_{b_{(0,0)}})^2\bigr)\Bigr)
	\end{aligned}\\
	=&c_{(0,1)}\tag{C($0,1$)}\label{eq:C01}
\end{align*}
hold in $\mathsf{F}(\mathsf{U}P+\mathsf{U}P)$ and the equations
\begin{align*}
	\mu^{(0,0)}_1\bigl(D^{(0,0)}_1,\ldots,D^{(0,0)}_{a_{(0,0)}}\bigr)
	&=\lambda^{(0,0)}_1\bigl(D^{(0,0)}_1,\ldots,D^{(0,0)}_{a_{(0,0)}}\bigr),\tag{M($0,0$)}\label{eq:M00}\\
	\mu^{(0,0)}_2\bigl(E^{(0,0)}_1,\ldots,E^{(0,0)}_{b_{(0,0)}}\bigr)
	&=\lambda^{(0,0)}_2\bigl(E^{(0,0)}_1,\ldots,E^{(0,0)}_{b_{(0,0)}}\bigr)\tag{N($0,0$)}\label{eq:N00}
\end{align*}
hold in~$P$; the equations
{\allowdisplaybreaks
\begin{align*}
		&c_{(0,j)}\\*
		=&\tau_{(0,j)}
		\begin{aligned}[t]
			\Bigl(&(D^{(0,j)}_1)^1,\ldots,(D^{(0,j)}_{a_{(0,j)}})^1,
			(E^{(0,j)}_1)^2,\ldots,(E^{(0,j)}_{b_{(0,j)}})^2,\\*
			&\mu^{(0,j)}_1\bigl((D^{(0,j)}_1)^1,\ldots,(D^{(0,j)}_{a_{(0,j)}})^1\bigr),
			\mu^{(0,j)}_2\bigl((E^{(0,j)}_1)^2,\ldots,(E^{(0,j)}_{b_{(0,j)}})^2\bigr)\Bigr),
		\end{aligned}\tag{D($0,\j$)}\label{eq:D0j}\\[5ex]
		&\tau_{(0,j)}
		\begin{aligned}[t]
			\Bigl(&(D^{(0,j)}_1)^1,\ldots,(D^{(0,j)}_{a_{(0,j)}})^1,
			(E^{(0,j)}_1)^2,\ldots,(E^{(0,j)}_{b_{(0,j)}})^2,\\*
			&\lambda^{(0,j)}_1\bigl((D^{(0,j)}_1)^1,\ldots,(D^{(0,j)}_{a_{(0,j)}})^1\bigr),
			\lambda^{(0,j)}_2\bigl((E^{(0,j)}_1)^2,\ldots,(E^{(0,j)}_{b_{(0,j)}})^2\bigr)\Bigr)
		\end{aligned}\\*
		=&c_{(0,j+1)}\tag{C($0,\jp$)}\label{eq:C0j}
\end{align*}}%
and 
\begin{align*}
	\mu^{(0,j)}_1\bigl(D^{(0,j)}_1,\ldots,D^{(0,j)}_{a_{(0,j)}}\bigr)
	&=\lambda^{(0,j)}_1\bigl(D^{(0,j)}_1,\ldots,D^{(0,j)}_{a_{(0,j)}}\bigr),\tag{M($0,\j$)}\label{eq:M0j}\\
	\mu^{(0,j)}_2\bigl(E^{(0,j)}_1,\ldots,E^{(0,j)}_{b_{(0,j)}}\bigr)
	&=\lambda^{(0,j)}_2\bigl(E^{(0,j)}_1,\ldots,E^{(0,j)}_{b_{(0,j)}}\bigr)\tag{N($0,\j$)}\label{eq:N0j}
\end{align*}
are satisfied for all $j\in\{1,\ldots,\ell_0-1\}$; and moreover the equations
{\allowdisplaybreaks
\begin{align*}
		&c_{(0,\ell_0)}\\*
		=&\tau_{(0,\ell_0)}
		\begin{aligned}[t]
			\Bigl(&(D^{(0,\ell_0)}_1)^1,\ldots,(D^{(0,\ell_0)}_{a_{(0,\ell_0)}})^1,
			(E^{(0,\ell_0)}_1)^2,\ldots,(E^{(0,\ell_0)}_{b_{(0,\ell_0)}})^2,\\*
			&\mu^{(0,\ell_0)}_1\bigl((D^{(0,\ell_0)}_1)^1,\ldots,(D^{(0,\ell_0)}_{a_{(0,\ell_0)}})^1\bigr),
			\mu^{(0,\ell_0)}_2\bigl((E^{(0,\ell_0)}_1)^2,\ldots,(E^{(0,\ell_0)}_{b_{(0,\ell_0)}})^2\bigr)\Bigr),
		\end{aligned}\tag{D($0,\ell_0$)}\label{eq:D0l0}\\[5ex]
		&\tau_{(0,\ell_0)}
		\begin{aligned}[t]
			\Bigl(&(D^{(0,\ell_0)}_1)^1,\ldots,(D^{(0,\ell_0)}_{a_{(0,\ell_0)}})^1,
			(E^{(0,\ell_0)}_1)^2,\ldots,(E^{(0,\ell_0)}_{b_{(0,\ell_0)}})^2,\\*
			&\lambda^{(0,\ell_0)}_1\bigl((D^{(0,\ell_0)}_1)^1,\ldots,(D^{(0,\ell_0)}_{a_{(0,\ell_0)}})^1\bigr),
			\lambda^{(0,\ell_0)}_2\bigl((E^{(0,\ell_0)}_1)^2,\ldots,(E^{(0,\ell_0)}_{b_{(0,\ell_0)}})^2\bigr)\Bigr)
		\end{aligned}\\*
		=&\begin{aligned}[t]
			t'_1\Bigl(
			&(y,x)^1,(f_1,g_1)^1,\ldots,(f_k,g_k)^1,\\*
			&({p_1(x,x,y)},{p_1(x,y,y)})^1,\ldots,({p_m(x,x,y)},{p_m(x,y,y)})^1,\\*
			&({p_1(x,x,y)},{p_1(x,y,y)})^1,\ldots,({p_m(x,x,y)},{p_m(x,y,y)})^1,\\*
			&(y,x)^2,(f_1,g_1)^2,\ldots,(f_k,g_k)^2,\\*
			&({p_1(x,x,y)},{p_1(x,y,y)})^2,\ldots,({p_m(x,x,y)},{p_m(x,y,y)})^2,\\*
			&({p_1(x,x,y)},{p_1(x,y,y)})^2,\ldots,({p_m(x,x,y)},{p_m(x,y,y)})^2\Bigr)
		\end{aligned}\tag{T($1$)}\label{eq:T1}
\end{align*}}%
and
\begin{align*}
	\mu^{(0,\ell_0)}_1\bigl(D^{(0,\ell_0)}_1,\ldots,D^{(0,\ell_0)}_{a_{(0,\ell_0)}}\bigr)
	&=\lambda^{(0,\ell_0)}_1\bigl(D^{(0,\ell_0)}_1,\ldots,D^{(0,\ell_0)}_{a_{(0,\ell_0)}}\bigr),\tag{M($0,\ell_0$)}\label{eq:M0l0}\\
	\mu^{(0,\ell_0)}_2\bigl(E^{(0,\ell_0)}_1,\ldots,E^{(0,\ell_0)}_{b_{(0,\ell_0)}}\bigr)
	&=\lambda^{(0,\ell_0)}_2\bigl(E^{(0,\ell_0)}_1,\ldots,E^{(0,\ell_0)}_{b_{(0,\ell_0)}}\bigr)\tag{N($0,\ell_0$)}\label{eq:N0l0}
\end{align*}
hold. For $i\in\{1,\ldots,n-1\}$, Equation~\eqref{eq:bi} means that there exist an integer $\ell_i\geqslant 0$ and elements $c_{(i,1)},\ldots,c_{(i,\ell_i)}\in \mathsf{F}(\mathsf{U}P+\mathsf{U}P)$ such that
{\allowdisplaybreaks
\begin{equation*}
	\begin{split}
		&
		\begin{aligned}[t] 
			t'_i\bigl(
			&(y,x)^1,(f_1,g_1)^1,\ldots,(f_k,g_k)^1,\\*
			&({p_1(x,x,y)},{p_1(x,y,y)})^1,\ldots,({p_m(x,x,y)},{p_m(x,y,y)})^1,\\*
			&({p_1(x,x,y)},{p_1(x,y,y)})^2,\ldots,({p_m(x,x,y)},{p_m(x,y,y)})^2,\\*
			&(y,x)^2,(f_1,g_1)^2,\ldots,(f_k,g_k)^2,\\*
			&({p_1(x,x,y)},{p_1(x,y,y)})^2,\ldots,({p_m(x,x,y)},{p_m(x,y,y)})^2,\\*
			&({p_1(x,x,y)},{p_1(x,y,y)})^1,\ldots,({p_m(x,x,y)},{p_m(x,y,y)})^1\bigr)\sim c_{(i,1)},
		\end{aligned}
	\\
	&c_{(i,1)}\sim c_{(i,2)},\ldots,c_{(i,\ell_i-1)}\sim c_{(i,\ell_i)},\\
	&c_{(i,\ell_i)}\sim
		 \begin{aligned}[t] 
		 	t'_{i+1}\bigl(
			&(y,x)^1,(f_1,g_1)^1,\ldots,(f_k,g_k)^1,\\*
			&({p_1(x,x,y)},{p_1(x,y,y)})^1,\ldots,({p_m(x,x,y)},{p_m(x,y,y)})^1,\\*
			&({p_1(x,x,y)},{p_1(x,y,y)})^1,\ldots,({p_m(x,x,y)},{p_m(x,y,y)})^1,\\*
			&(y,x)^2,(f_1,g_1)^2,\ldots,(f_k,g_k)^2,\\*
			&({p_1(x,x,y)},{p_1(x,y,y)})^2,\ldots,({p_m(x,x,y)},{p_m(x,y,y)})^2,\\*
			&({p_1(x,x,y)},{p_1(x,y,y)})^2,\ldots,({p_m(x,x,y)},{p_m(x,y,y)})^2\bigr).
		 \end{aligned}
	\end{split}
\end{equation*}}
More precisely, there exist, for all $j\in\{0,\ldots,\ell_i\}$, integers $a_{(i,j)}, b_{(i,j)}\geqslant 0$, elements $D^{(i,j)}_1,\ldots,D^{(i,j)}_{a_{(i,j)}}$, $E^{(i,j)}_1,\ldots,E^{(i,j)}_{b_{(i,j)}}\in P$, $a_{(i,j)}$-ary terms $\mu^{(i,j)}_1,\lambda^{(i,j)}_1$, $b_{(i,j)}$-ary terms $\mu^{(i,j)}_2$, $\lambda^{(i,j)}_2$ and $(a_{(i,j)}+b_{(i,j)}+2)$-ary terms $\tau_{(i,j)}$ such that the equations
\begin{align*}
		&
		\begin{aligned}[t] 
			t'_i\Bigl(
			&(y,x)^1,(f_1,g_1)^1,\ldots,(f_k,g_k)^1,\\*
			&({p_1(x,x,y)},{p_1(x,y,y)})^1,\ldots,({p_m(x,x,y)},{p_m(x,y,y)})^1,\\*
			&({p_1(x,x,y)},{p_1(x,y,y)})^2,\ldots,({p_m(x,x,y)},{p_m(x,y,y)})^2,\\*
			&(y,x)^2,(f_1,g_1)^2,\ldots,(f_k,g_k)^2,\\*
			&({p_1(x,x,y)},{p_1(x,y,y)})^2,\ldots,({p_m(x,x,y)},{p_m(x,y,y)})^2,\\*
			&({p_1(x,x,y)},{p_1(x,y,y)})^1,\ldots,({p_m(x,x,y)},{p_m(x,y,y)})^1\Bigr)
		\end{aligned}\\
		=&\tau_{(i,0)}
		\begin{aligned}[t]
			\Bigl(&(D^{(i,0)}_1)^1,\ldots,(D^{(i,0)}_{a_{(i,0)}})^1,
			(E^{(i,0)}_1)^2,\ldots,(E^{(i,0)}_{b_{(i,0)}})^2,\\
			&\mu^{(i,0)}_1\bigl((D^{(i,0)}_1)^1,\ldots,(D^{(i,0)}_{a_{(i,0)}})^1\bigr),
			\mu^{(i,0)}_2\bigl((E^{(i,0)}_1)^2,\ldots,(E^{(i,0)}_{b_{(i,0)}})^2\bigr)\Bigr)
		\end{aligned}\tag{S($\i$)}\label{eq:Si}
\end{align*}
and
\begin{align*}
		&\tau_{(i,0)}
		\begin{aligned}[t]
			\Bigl(&(D^{(i,0)}_1)^1,\ldots,(D^{(i,0)}_{a_{(i,0)}})^1,
			(E^{(i,0)}_1)^2,\ldots,(E^{(i,0)}_{b_{(i,0)}})^2,\\
			&\lambda^{(i,0)}_1\bigl((D^{(i,0)}_1)^1,\ldots,(D^{(i,0)}_{a_{(i,0)}})^1\bigr),
			\lambda^{(i,0)}_2\bigl((E^{(i,0)}_1)^2,\ldots,(E^{(i,0)}_{b_{(i,0)}})^2\bigr)\Bigr)
		\end{aligned}\\
		=&c_{(i,1)}\tag{C($\i,1$)}\label{eq:Ci1}
\end{align*}
hold in $\mathsf{F}(\mathsf{U}P+\mathsf{U}P)$ and the equations
\begin{align*}
	\mu^{(i,0)}_1\bigl(D^{(i,0)}_1,\ldots,D^{(i,0)}_{a_{(i,0)}}\bigr)
	&=\lambda^{(i,0)}_1\bigl(D^{(i,0)}_1,\ldots,D^{(i,0)}_{a_{(i,0)}}\bigr),\tag{M($\i,0$)}\label{eq:Mi0}\\
	\mu^{(i,0)}_2\bigl(E^{(i,0)}_1,\ldots,E^{(i,0)}_{b_{(i,0)}}\bigr)
	&=\lambda^{(i,0)}_2\bigl(E^{(i,0)}_1,\ldots,E^{(i,0)}_{b_{(i,0)}}\bigr)\tag{N($\i,0$)}\label{eq:Ni0}
\end{align*}
hold in~$P$; the equations
{\allowdisplaybreaks
\begin{align*}
		&c_{(i,j)}\\*
		=&\tau_{(i,j)}
		\begin{aligned}[t]
			\Bigl(&(D^{(i,j)}_1)^1,\ldots,(D^{(i,j)}_{a_{(i,j)}})^1,
			(E^{(i,j)}_1)^2,\ldots,(E^{(i,j)}_{b_{(i,j)}})^2,\\*
			&\mu^{(i,j)}_1\bigl((D^{(i,j)}_1)^1,\ldots,(D^{(i,j)}_{a_{(i,j)}})^1\bigr),
			\mu^{(i,j)}_2\bigl((E^{(i,j)}_1)^2,\ldots,(E^{(i,j)}_{b_{(i,j)}})^2\bigr)\Bigr),
		\end{aligned}\tag{D($\i,\j$)}\label{eq:Dij}\\[5ex]
		&\tau_{(i,j)}
		\begin{aligned}[t]
			\Bigl(&(D^{(i,j)}_1)^1,\ldots,(D^{(i,j)}_{a_{(i,j)}})^1,
			(E^{(i,j)}_1)^2,\ldots,(E^{(i,j)}_{b_{(i,j)}})^2,\\*
			&\lambda^{(i,j)}_1\bigl((D^{(i,j)}_1)^1,\ldots,(D^{(i,j)}_{a_{(i,j)}})^1\bigr),
			\lambda^{(i,j)}_2\bigl((E^{(i,j)}_1)^2,\ldots,(E^{(i,j)}_{b_{(i,j)}})^2\bigr)\Bigr)
		\end{aligned}\\*
		=&c_{(i,j+1)}\tag{C($\i,\jp$)}\label{eq:Cij}
\end{align*}}
and
\begin{align*}
	\mu^{(i,j)}_1\bigl(D^{(i,j)}_1,\ldots,D^{(i,j)}_{a_{(i,j)}}\bigr)
	&=\lambda^{(i,j)}_1\bigl(D^{(i,j)}_1,\ldots,D^{(i,j)}_{a_{(i,j)}}\bigr),\tag{M($\i,\j$)}\label{eq:Mij}\\
	\mu^{(i,j)}_2\bigl(E^{(i,j)}_1,\ldots,E^{(i,j)}_{b_{(i,j)}}\bigr)
	&=\lambda^{(i,j)}_2\bigl(E^{(i,j)}_1,\ldots,E^{(i,j)}_{b_{(i,j)}}\bigr)\tag{N($\i,\j$)}\label{eq:Nij}
\end{align*}
are satisfied for all $j\in\{1,\ldots,\ell_i-1\}$; and moreover the equations
{\allowdisplaybreaks
\begin{align*}
		&c_{(i,\ell_i)}\\*
		=&\tau_{(i,\ell_i)}
		\begin{aligned}[t]
			\Bigl(&(D^{(i,\ell_i)}_1)^1,\ldots,(D^{(i,\ell_i)}_{a_{(i,\ell_i)}})^1,
			(E^{(i,\ell_i)}_1)^2,\ldots,(E^{(i,\ell_i)}_{b_{(i,\ell_i)}})^2,\\*
			&\mu^{(i,\ell_i)}_1\bigl((D^{(i,\ell_i)}_1)^1,\ldots,(D^{(i,\ell_i)}_{a_{(i,\ell_i)}})^1\bigr),
			\mu^{(i,\ell_i)}_2\bigl((E^{(i,\ell_i)}_1)^2,\ldots,(E^{(i,\ell_i)}_{b_{(i,\ell_i)}})^2\bigr)\Bigr),
		\end{aligned}\tag{D($\i,\ell_{\i}$)}\label{eq:Dili}\\[5ex]
		&\tau_{(i,\ell_i)}
		\begin{aligned}[t]
			\Bigl(&(D^{(i,\ell_i)}_1)^1,\ldots,(D^{(i,\ell_i)}_{a_{(i,\ell_i)}})^1,
			(E^{(i,\ell_i)}_1)^2,\ldots,(E^{(i,\ell_i)}_{b_{(i,\ell_i)}})^2,\\*
			&\lambda^{(i,\ell_i)}_1\bigl((D^{(i,\ell_i)}_1)^1,\ldots,(D^{(i,\ell_i)}_{a_{(i,\ell_i)}})^1\bigr),
			\lambda^{(i,\ell_i)}_2\bigl((E^{(i,\ell_i)}_1)^2,\ldots,(E^{(i,\ell_i)}_{b_{(i,\ell_i)}})^2\bigr)\Bigr)
		\end{aligned}\\*
		=&
		\begin{aligned}[t]
			t'_{i+1}\Bigl(
			&(y,x)^1,(f_1,g_1)^1,\ldots,(f_k,g_k)^1,\\*
			&({p_1(x,x,y)},{p_1(x,y,y)})^1,\ldots,({p_m(x,x,y)},{p_m(x,y,y)})^1,\\*
			&({p_1(x,x,y)},{p_1(x,y,y)})^1,\ldots,({p_m(x,x,y)},{p_m(x,y,y)})^1,\\*
			&(y,x)^2,(f_1,g_1)^2,\ldots,(f_k,g_k)^2,\\*
			&({p_1(x,x,y)},{p_1(x,y,y)})^2,\ldots,({p_m(x,x,y)},{p_m(x,y,y)})^2,\\*
			&({p_1(x,x,y)},{p_1(x,y,y)})^2,\ldots,({p_m(x,x,y)},{p_m(x,y,y)})^2\Bigr)
	\end{aligned}\tag{T($\ip$)}\label{eq:Ti}
\end{align*}}%
and
\begin{align*}
	\mu^{(i,\ell_i)}_1\bigl(D^{(i,\ell_i)}_1,\ldots,D^{(i,\ell_i)}_{a_{(i,\ell_i)}}\bigr)
	&=\lambda^{(i,\ell_i)}_1\bigl(D^{(i,\ell_i)}_1,\ldots,D^{(i,\ell_i)}_{a_{(i,\ell_i)}}\bigr),\tag{M($\i,\ell_{\i}$)}\label{eq:Mili}\\
	\mu^{(i,\ell_i)}_2\bigl(E^{(i,\ell_i)}_1,\ldots,E^{(i,\ell_i)}_{b_{(i,\ell_i)}}\bigr)
	&=\lambda^{(i,\ell_i)}_2\bigl(E^{(i,\ell_i)}_1,\ldots,E^{(i,\ell_i)}_{b_{(i,\ell_i)}}\bigr)\tag{N($\i,\ell_{\i}$)}\label{eq:Nili}
\end{align*}
hold. Equation~\eqref{eq:c} means that there exist an integer $\ell_n\geqslant 0$ and elements $c_{(n,1)},\ldots,c_{(n,\ell_n)}$ in ${\mathsf{F}(\mathsf{U}P+\mathsf{U}P)}$ such that
{\allowdisplaybreaks
\begin{equation*}
	\begin{split}
		&\begin{aligned}[t] 
			t'_n\bigl(
			&(y,x)^1,(f_1,g_1)^1,\ldots,(f_k,g_k)^1,\\*
			&({p_1(x,x,y)},{p_1(x,y,y)})^1,\ldots,({p_m(x,x,y)},{p_m(x,y,y)})^1,\\*
			&({p_1(x,x,y)},{p_1(x,y,y)})^2,\ldots,({p_m(x,x,y)},{p_m(x,y,y)})^2,\\*
			&(y,x)^2,(f_1,g_1)^2,\ldots,(f_k,g_k)^2,\\*
			&({p_1(x,x,y)},{p_1(x,y,y)})^2,\ldots,({p_m(x,x,y)},{p_m(x,y,y)})^2,\\*
			&({p_1(x,x,y)},{p_1(x,y,y)})^1,\ldots,({p_m(x,x,y)},{p_m(x,y,y)})^1\bigr)\sim c_{(n,1)},
		\end{aligned}\\
		&c_{(n,1)}\sim c_{(n,2)},\ldots,c_{(n,\ell_n-1)}\sim c_{(n,\ell_n)},\\*
		&c_{(n,\ell_n)}\sim (y,x)^2.
	\end{split}
\end{equation*}}
More explicitly, this means that there exist, for all $j\in\{0,\ldots,\ell_n\}$, integers $a_{(n,j)}, b_{(n,j)}\geqslant 0$, elements $D^{(n,j)}_1,\ldots,D^{(n,j)}_{a_{(n,j)}},E^{(n,j)}_1,\ldots,E^{(n,j)}_{b_{(n,j)}}\in P$, $a_{(n,j)}$-ary terms $\mu^{(n,j)}_1,\lambda^{(n,j)}_1$, $b_{(n,j)}$-ary terms $\mu^{(n,j)}_2,\lambda^{(n,j)}_2$ and $(a_{(n,j)}+b_{(n,j)}+2)$-ary terms $\tau_{(n,j)}$ such that the equations
\begin{align*}
		&\begin{aligned}[t] 
			t'_n\Bigl(
			&(y,x)^1,(f_1,g_1)^1,\ldots,(f_k,g_k)^1,\\*
			&({p_1(x,x,y)},{p_1(x,y,y)})^1,\ldots,({p_m(x,x,y)},{p_m(x,y,y)})^1,\\*
			&({p_1(x,x,y)},{p_1(x,y,y)})^2,\ldots,({p_m(x,x,y)},{p_m(x,y,y)})^2,\\*
			&(y,x)^2,(f_1,g_1)^2,\ldots,(f_k,g_k)^2,\\*
			&({p_1(x,x,y)},{p_1(x,y,y)})^2,\ldots,({p_m(x,x,y)},{p_m(x,y,y)})^2,\\*
			&({p_1(x,x,y)},{p_1(x,y,y)})^1,\ldots,({p_m(x,x,y)},{p_m(x,y,y)})^1\Bigr)
		\end{aligned}\\
		=&\tau_{(n,0)}
		\begin{aligned}[t]
			\Bigl(&(D^{(n,0)}_1)^1,\ldots,(D^{(n,0)}_{a_{(n,0)}})^1,
			(E^{(n,0)}_1)^2,\ldots,(E^{(n,0)}_{b_{(n,0)}})^2,\\
			&\mu^{(n,0)}_1\bigl((D^{(n,0)}_1)^1,\ldots,(D^{(n,0)}_{a_{(n,0)}})^1\bigr),
			\mu^{(n,0)}_2\bigl((E^{(n,0)}_1)^2,\ldots,(E^{(n,0)}_{b_{(n,0)}})^2\bigr)\Bigr)
		\end{aligned}\tag{S($n$)}\label{eq:Sn}
\end{align*}
and
\begin{align*}
		&\tau_{(n,0)}
		\begin{aligned}[t]
			\Bigl(&(D^{(n,0)}_1)^1,\ldots,(D^{(n,0)}_{a_{(n,0)}})^1,
			(E^{(n,0)}_1)^2,\ldots,(E^{(n,0)}_{b_{(n,0)}})^2,\\
			&\lambda^{(n,0)}_1\bigl((D^{(n,0)}_1)^1,\ldots,(D^{(n,0)}_{a_{(n,0)}})^1\bigr),
			\lambda^{(n,0)}_2\bigl((E^{(n,0)}_1)^2,\ldots,(E^{(n,0)}_{b_{(n,0)}})^2\bigr)\Bigr)
		\end{aligned}\\
		=&c_{(n,1)}\tag{C($n,1$)}\label{eq:Cn1}
\end{align*}
hold in $\mathsf{F}(\mathsf{U}P+\mathsf{U}P)$ and the equations
\begin{align*}
	\mu^{(n,0)}_1\bigl(D^{(n,0)}_1,\ldots,D^{(n,0)}_{a_{(n,0)}}\bigr)
	&=\lambda^{(n,0)}_1\bigl(D^{(n,0)}_1,\ldots,D^{(n,0)}_{a_{(n,0)}}\bigr),\tag{M($n,0$)}\label{eq:Mn0}\\
	\mu^{(n,0)}_2\bigl(E^{(n,0)}_1,\ldots,E^{(n,0)}_{b_{(n,0)}}\bigr)
	&=\lambda^{(n,0)}_2\bigl(E^{(n,0)}_1,\ldots,E^{(n,0)}_{b_{(n,0)}}\bigr)\tag{N($n,0$)}\label{eq:Nn0}
\end{align*}
hold in~$P$; the equations
{\allowdisplaybreaks
\begin{align*}
		&c_{(n,j)}\\*
		=&\tau_{(n,j)}
		\begin{aligned}[t]
			\Bigl(&(D^{(n,j)}_1)^1,\ldots,(D^{(n,j)}_{a_{(n,j)}})^1,
			(E^{(n,j)}_1)^2,\ldots,(E^{(n,j)}_{b_{(n,j)}})^2,\\*
			&\mu^{(n,j)}_1\bigl((D^{(n,j)}_1)^1,\ldots,(D^{(n,j)}_{a_{(n,j)}})^1\bigr),
			\mu^{(n,j)}_2\bigl((E^{(n,j)}_1)^2,\ldots,(E^{(n,j)}_{b_{(n,j)}})^2\bigr)\Bigr),
		\end{aligned}\tag{D($n,\j$)}\label{eq:Dnj}\\[5ex]
		&\tau_{(n,j)}
		\begin{aligned}[t]
			\Bigl(&(D^{(n,j)}_1)^1,\ldots,(D^{(n,j)}_{a_{(n,j)}})^1,
			(E^{(n,j)}_1)^2,\ldots,(E^{(n,j)}_{b_{(n,j)}})^2,\\*
			&\lambda^{(n,j)}_1\bigl((D^{(n,j)}_1)^1,\ldots,(D^{(n,j)}_{a_{(n,j)}})^1\bigr),
			\lambda^{(n,j)}_2\bigl((E^{(n,j)}_1)^2,\ldots,(E^{(n,j)}_{b_{(n,j)}})^2\bigr)\Bigr)
		\end{aligned}\\*
		=&c_{(n,j+1)}\tag{C($n,\jp$)}\label{eq:Cnj}	
\end{align*}}%
and
\begin{align*}
	\mu^{(n,j)}_1\bigl(D^{(n,j)}_1,\ldots,D^{(n,j)}_{a_{(n,j)}}\bigr)
	&=\lambda^{(n,j)}_1\bigl(D^{(n,j)}_1,\ldots,D^{(n,j)}_{a_{(n,j)}}\bigr),\tag{M($n,\j$)}\label{eq:Mnj}\\
	\mu^{(n,j)}_2\bigl(E^{(n,j)}_1,\ldots,E^{(n,j)}_{b_{(n,j)}}\bigr)
	&=\lambda^{(n,j)}_2\bigl(E^{(n,j)}_1,\ldots,E^{(n,j)}_{b_{(n,j)}}\bigr)\tag{N($n,\j$)}\label{eq:Nnj}
\end{align*}
are satisfied for all $j\in\{1,\ldots,\ell_n-1\}$; and moreover the equations
{\allowdisplaybreaks
\begin{align*}
		&c_{(n,\ell_n)}\\*
		=&\tau_{(n,\ell_n)}
		\begin{aligned}[t]
			\Bigl(&(D^{(n,\ell_n)}_1)^1,\ldots,(D^{(n,\ell_n)}_{a_{(n,\ell_n)}})^1,
			(E^{(n,\ell_n)}_1)^2,\ldots,(E^{(n,\ell_n)}_{b_{(n,\ell_n)}})^2,\\*
			&\mu^{(n,\ell_n)}_1\bigl((D^{(n,\ell_n)}_1)^1,\ldots,(D^{(n,\ell_n)}_{a_{(n,\ell_n)}})^1\bigr),
			\mu^{(n,\ell_n)}_2\bigl((E^{(n,\ell_n)}_1)^2,\ldots,(E^{(n,\ell_n)}_{b_{(n,\ell_n)}})^2\bigr)\Bigr),
		\end{aligned}\tag{D($n,\ell_n$)}\label{eq:Dnln}\\[5ex]
		&\tau_{(n,\ell_n)}
		\begin{aligned}[t]
			\Bigl(&(D^{(n,\ell_n)}_1)^1,\ldots,(D^{(n,\ell_n)}_{a_{(n,\ell_n)}})^1,
			(E^{(n,\ell_n)}_1)^2,\ldots,(E^{(n,\ell_n)}_{b_{(n,\ell_n)}})^2,\\*
			&\lambda^{(n,\ell_n)}_1\bigl((D^{(n,\ell_n)}_1)^1,\ldots,(D^{(n,\ell_n)}_{a_{(n,\ell_n)}})^1\bigr),
			\lambda^{(n,\ell_n)}_2\bigl((E^{(n,\ell_n)}_1)^2,\ldots,(E^{(n,\ell_n)}_{b_{(n,\ell_n)}})^2\bigr)\Bigr)
		\end{aligned}\\*
		=&(y,x)^2 \tag{B2}\label{eq:B2}
\end{align*}}
and
\begin{align*}
	\mu^{(n,\ell_n)}_1\bigl(D^{(n,\ell_n)}_1,\ldots,D^{(n,\ell_n)}_{a_{(n,\ell_n)}}\bigr)
	&=\lambda^{(n,\ell_n)}_1\bigl(D^{(n,\ell_n)}_1,\ldots,D^{(n,\ell_n)}_{a_{(n,\ell_n)}}\bigr),\tag{M($n,\ell_n$)}\label{eq:Mnln}\\
	\mu^{(n,\ell_n)}_2\bigl(E^{(n,\ell_n)}_1,\ldots,E^{(n,\ell_n)}_{b_{(n,\ell_n)}}\bigr)
	&=\lambda^{(n,\ell_n)}_2\bigl(E^{(n,\ell_n)}_1,\ldots,E^{(n,\ell_n)}_{b_{(n,\ell_n)}}\bigr)\tag{N($n,\ell_n$)}\label{eq:Nnln}
\end{align*}
hold. Similar to Diagram~\eqref{diagr:smalltrail}, we display the Equations~\eqref{eq:B1}, \eqref{eq:B2}, $\eqrefip{eq:Ti}{i}$, \eqref{eq:Si}, $\eqrefijp{eq:Cij}{i}{j}$ and \eqref{eq:Dij} in the diagram
\begin{footnotesize}
\begin{align*}
	\eqref{eq:a}
	&\begin{cases}
		\xymatrixcolsep{3.7mm}\xymatrix{
			(y,x)^1					\ar@{=}[rr]^{\eqref{eq:B1}}
			&&\tau_{(0,0)}[\mu]		\ar@{.}[r]
			&\tau_{(0,0)}[\lambda]	\ar@{=}[rr]^-{\eqref{eq:C01}}
			&&c_{(0,1)}				\ar@{=}[rr]^-{\eqrefj{eq:D0j}{1}}
			&&\tau_{(0,1)}[\mu]		\ar@{.}[r]
			&\tau_{(0,1)}[\lambda]	\ar@{=}[rrr]^-{\eqrefjp{eq:C0j}{2}}
			&&&c_{(0,2)}				\ar@{}[d]^-{\vdots}
			\\
			t'_1[1,2]				\ar@{=}[rr]^-{\eqref{eq:T1}}
			&&\tau_{(0,\ell_0)}[\lambda] \ar@{.}[r]
			&\tau_{(0,\ell_0)}[\mu]	\ar@{=}[rr]^-{\eqref{eq:D0l0}}
			&&c_{(0,\ell_0)}		\ar@{=}[rr]^-{\eqrefjp{eq:C0j}{\ell_0}}
			&&\tau_{(0,\ell_0-1)}[\lambda] \ar@{.}[r]
			&\tau_{(0,\ell_0-1)}[\mu]\ar@{=}[rrr]^-{\eqrefj{eq:D0j}{\ell_0-1}}
			&&&c_{(0,\ell_0-1)}
		}
	\end{cases}\\
	\eqref{eq:bi}
	&\begin{cases}
		\xymatrixcolsep{3.7mm}\xymatrix{
			t'_i[2,1]					\ar@{=}[rr]^{\eqref{eq:Si}}
			&&\tau_{(i,0)}[\mu]			\ar@{.}[r]
			&\tau_{(i,0)}[\lambda]	\ar@{=}[rr]^-{\eqref{eq:Ci1}}
			&&c_{(i,1)}				\ar@{=}[rr]^-{\eqrefij{eq:Dij}{i}{1}}
			&&\tau_{(i,1)}[\mu]		\ar@{.}[r]
			&\tau_{(i,1)}[\lambda]	\ar@{=}[rrr]^-{\eqrefijp{eq:Cij}{i}{2}}
			&&&c_{(i,2)}				\ar@{}[d]^-{\vdots}
			\\
			t'_{i+1}[1,2]				\ar@{=}[rr]^-{\eqref{eq:Ti}}
			&&\tau_{(i,\ell_i)}[\lambda] \ar@{.}[r]
			&\tau_{(i,\ell_i)}[\mu]	\ar@{=}[rr]^-{\eqref{eq:Dili}}
			&&c_{(i,\ell_i)}		\ar@{=}[rr]^-{\eqrefijp{eq:Cij}{i}{\ell_i}}
			&&\tau_{(i,\ell_i-1)}[\lambda] \ar@{.}[r]
			&\tau_{(i,\ell_i-1)}[\mu]\ar@{=}[rrr]^-{\eqrefij{eq:Dij}{i}{\ell_i-1}}
			&&&c_{(i,\ell_i-1)}
		}
	\end{cases}\\
	\eqref{eq:c}
	&\begin{cases}
		\xymatrixcolsep{3.7mm}\xymatrix{
			t'_n[2,1]					\ar@{=}[rr]^{\eqref{eq:Sn}}
			&&\tau_{(n,0)}[\mu]			\ar@{.}[r]
			&\tau_{(n,0)}[\lambda]	\ar@{=}[rr]^-{\eqref{eq:Cn1}}
			&&c_{(n,1)}				\ar@{=}[rr]^-{\eqrefj{eq:Dnj}{1}}
			&&\tau_{(n,1)}[\mu]		\ar@{.}[r]	
			&\tau_{(n,1)}[\lambda]	\ar@{=}[rrr]^-{\eqrefjp{eq:Cnj}{2}}
			&&&c_{(n,2)}				\ar@{}[d]^-{\vdots}
			\\
			(y,x)^2				\ar@{=}[rr]^-{\eqref{eq:B2}}
			&&\tau_{(n,\ell_n)}[\lambda]	\ar@{.}[r]
			&\tau_{(n,\ell_n)}[\mu]	\ar@{=}[rr]^-{\eqref{eq:Dnln}}
			&&c_{(n,\ell_n)}		\ar@{=}[rr]^-{\eqrefjp{eq:Cnj}{\ell_n}}
			&&\tau_{(n,\ell_n-1)}[\lambda] \ar@{.}[r]
			&\tau_{(n,\ell_n-1)}[\mu]\ar@{=}[rrr]^-{\eqrefj{eq:Dnj}{\ell_n-1}}
			&&&c_{(n,\ell_n-1)},
		}
	\end{cases}
\end{align*}
\end{footnotesize}
where
{\allowdisplaybreaks
\begin{align*}
	t'_i[1,2]
	&:=
	\begin{aligned}[t] 
		t'_i\Bigl(
			&(y,x)^1,(f_1,g_1)^1,\ldots,(f_k,g_k)^1,\\*
			&({p_1(x,x,y)},{p_1(x,y,y)})^1,\ldots,({p_m(x,x,y)},{p_m(x,y,y)})^1,\\*
			&({p_1(x,x,y)},{p_1(x,y,y)})^1,\ldots,({p_m(x,x,y)},{p_m(x,y,y)})^1,\\*
			&(y,x)^2,(f_1,g_1)^2,\ldots,(f_k,g_k)^2,\\*
			&({p_1(x,x,y)},{p_1(x,y,y)})^2,\ldots,({p_m(x,x,y)},{p_m(x,y,y)})^2,\\*
			&({p_1(x,x,y)},{p_1(x,y,y)})^2,\ldots,({p_m(x,x,y)},{p_m(x,y,y)})^2\Bigr),
	\end{aligned} 
	\\
	t'_i[2,1]
	&:=
	\begin{aligned}[t] 
		t'_i\Bigl(
			&(y,x)^1,(f_1,g_1)^1,\ldots,(f_k,g_k)^1,\\*
			&({p_1(x,x,y)},{p_1(x,y,y)})^1,\ldots,({p_m(x,x,y)},{p_m(x,y,y)})^1,\\*
			&({p_1(x,x,y)},{p_1(x,y,y)})^2,\ldots,({p_m(x,x,y)},{p_m(x,y,y)})^2,\\*
			&(y,x)^2,(f_1,g_1)^2,\ldots,(f_k,g_k)^2,\\*
			&({p_1(x,x,y)},{p_1(x,y,y)})^2,\ldots,({p_m(x,x,y)},{p_m(x,y,y)})^2,\\*
			&({p_1(x,x,y)},{p_1(x,y,y)})^1,\ldots,({p_m(x,x,y)},{p_m(x,y,y)})^1\Bigr),
	\end{aligned} 
	\\
	\tau_{(i,j)}[\mu]
	&:=
	\tau_{(i,j)}
	\begin{aligned}[t]
		\Bigl(&(D^{(i,j)}_1)^1,\ldots,(D^{(i,j)}_{a_{(i,j)}})^1,
		(E^{(i,j)}_1)^2,\ldots,(E^{(i,j)}_{b_{(i,j)}})^2,\\*
		&\mu^{(i,j)}_1\bigl((D^{(i,j)}_1)^1,\ldots,(D^{(i,j)}_{a_{(i,j)}})^1\bigr),
		\mu^{(i,j)}_2\bigl((E^{(i,j)}_1)^2,\ldots,(E^{(i,j)}_{b_{(i,j)}})^2\bigr)\Bigr),
	\end{aligned} 
	\\
	\tau_{(i,j)}[\lambda]
	&:=
	\tau_{(i,j)}
	\begin{aligned}[t]
		\Bigl(&(D^{(i,j)}_1)^1,\ldots,(D^{(i,j)}_{a_{(i,j)}})^1,
		(E^{(i,j)}_1)^2,\ldots,(E^{(i,j)}_{b_{(i,j)}})^2,\\*
		&\lambda^{(i,j)}_1\bigl((D^{(i,j)}_1)^1,\ldots,(D^{(i,j)}_{a_{(i,j)}})^1\bigr),
		\lambda^{(i,j)}_2\bigl((E^{(i,j)}_1)^2,\ldots,(E^{(i,j)}_{b_{(i,j)}})^2\bigr)\Bigr).
	\end{aligned}
\end{align*}}
Since the relation $\sim$ as in Proposition~\ref{prop:coproduct} is reflexive, we can assume, without loss of generality, that $\ell_0=\cdots=\ell_n=: \ell$. As before, by possibly artificially increasing the number of variables the terms of the form $\tau_{(i,j)}$, $\mu_1^{(i,j)}$, $\mu_2^{(i,j)}$, $\lambda_1^{(i,j)}$ or $\lambda_2^{(i,j)}$ depend on, we can also assume, without loss of generality, that we have integers $a,b\geqslant 0$ such that $a_{(i,j)}=a$ and $b_{(i,j)}=b$ for all $i\in\{0,\ldots,n\}$ and $j\in\{0,\ldots,\ell\}$, and omit the upper index from the elements $D^{(i,j)}_\alpha,E^{(i,j)}_\alpha$.
Moreover, we can, without loss of generality, assume that $a=b$ and $D_\alpha=E_\alpha$ for all $\alpha\in\{1,\ldots,a\}$. Furthermore, up to adding some $D_\alpha$'s to the list of $(f_j,g_j)$'s, and up to adding $(y,x)$, some $(f_j,g_j)$'s and some $(p_j(x,x,y),p_j(x,y,y))$'s to the list of $D_\alpha$'s, we can assume that $a=k+m+1$ and $D_1=(y,x)$, $D_2=(f_1,g_1),\ldots$, $D_{k+1}=(f_k,g_k)$ and $D_{k+2}=(p_1(x,x,y),p_1(x,y,y)),\ldots$, $D_{k+m+1}=(p_m(x,x,y),p_m(x,y,y))$, keeping the elements
$$(y,x),(f_1,g_1)\dots,(f_k,g_k),(p_1(x,x,y),p_1(x,y,y))\dots,(p_m(x,x,y),p_m(x,y,y))$$
of $P$ pairwise distinct. Thus, Equations~\eqref{eq:a}, \eqref{eq:bi} for all $i\in\{1,\ldots,n-1\}$ and \eqref{eq:c} imply the existence of an integer $\ell\geqslant 0$, terms $c_{(i,1)},\ldots,c_{(i,\ell)}$, ${(k+m+1)}$-ary terms $\mu_1^{(i,j)},\mu_2^{(i,j)},\lambda_1^{(i,j)},\lambda_2^{(i,j)}$ and $(2(k+m+2))$-ary terms $\tau_{(i,j)}$ for all $i\in\{0,\ldots,n\}$ and $j\in\{0,\ldots,\ell\}$ such that the following equalities hold. Using variables $u,v_1,\ldots,v_k,w_1,\ldots,w_m$ and $u',v'_1,\ldots,v'_k,w'_1,\ldots,w'_m$, Equations~\eqref{eq:B1} and \eqref{eq:B2} read
\begin{align*}
	u
	&=\tau_{(0,0)}(u,\vec{v},\vec{w},u',\vec{v}',\vec{w}',\mu^{(0,0)}_1(u,\vec{v},\vec{w}),\mu^{(0,0)}_2(u',\vec{v}',\vec{w}')),\tag{B1'}\label{eq:B1'}\\
	u'
	&=\tau_{(n,\ell)}(u,\vec{v},\vec{w},u',\vec{v}',\vec{w}',\lambda^{(n,\ell)}_1(u,\vec{v},\vec{w}),\lambda^{(n,\ell)}_2(u',\vec{v}',\vec{w}')),\tag{B2'}\label{eq:B2'}
\end{align*}
where we write $\vec{v}$ for $v_1\ldots,v_k$ and $\vec{w}$ for $w_1,\ldots,w_m$, and analogously for $\vec{v}'$ and~$\vec{w}'$. Equations~\eqref{eq:T1} and~\eqref{eq:Ti}, and Equations~\eqref{eq:Si} and \eqref{eq:Sn} read
\begin{align*}
	t'_{i}(u,\vec{v},\vec{w},\vec{w},u',\vec{v}',\vec{w}',\vec{w}')
	&=\tau_{(i-1,\ell)}(u,\vec{v},\vec{w},u',\vec{v}',\vec{w}',\lambda^{(i-1,\ell)}_1(u,\vec{v},\vec{w}),\lambda^{(i-1,\ell)}_2(u',\vec{v}',\vec{w}')),\tag{T($\i$)'}\label{eq:Ti'}\\
	t'_i(u,\vec{v},\vec{w},\vec{w}',u',\vec{v}',\vec{w}',\vec{w})
	&=\tau_{(i,0)}(u,\vec{v},\vec{w},u',\vec{v}',\vec{w}',\mu^{(i,0)}_1(u,\vec{v},\vec{w}),\mu^{(i,0)}_2(u',\vec{v}',\vec{w}')),\tag{S($\i$)'}\label{eq:Si'}
\end{align*}
for all $i\in\{1,\ldots,n\}$. Equations~\eqref{eq:C01}, \eqref{eq:C0j}, \eqref{eq:Ci1}, \eqref{eq:Cij}, \eqref{eq:Cn1}, \eqref{eq:Cnj}, and Equations~\eqref{eq:D0j}, \eqref{eq:D0l0}, \eqref{eq:Dij}, \eqref{eq:Dili}, \eqref{eq:Dnj}, \eqref{eq:Dnln} read
\begin{align*}
	c_{(i,j)}
	&=\tau_{(i,j-1)}(u,\vec{v},\vec{w},u',\vec{v}',\vec{w}',\lambda^{(i,j-1)}_1(u,\vec{v},\vec{w}),\lambda^{(i,j-1)}_2(u',\vec{v}',\vec{w}'))\tag{C($\i,\j$)'}\label{eq:Cij'},\\
	c_{(i,j)}
	&=\tau_{(i,j)}(u,\vec{v},\vec{w},u',\vec{v}',\vec{w}',\mu^{(i,j)}_1(u,\vec{v},\vec{w}),\mu^{(i,j)}_2(u',\vec{v}',\vec{w}')),\tag{D($\i,\j$)'}\label{eq:Dij'}
\end{align*}
for all $i\in\{0,\ldots,n\}$ and $j\in\{1,\ldots,\ell\}$ where the variables on which the $c_{(i,j)}$'s depend are left inexplicit. Equations~\eqref{eq:M00}, \eqref{eq:M0j}, \eqref{eq:M0l0}, \eqref{eq:Mi0}, \eqref{eq:Mij}, \eqref{eq:Mili}, \eqref{eq:Mn0}, \eqref{eq:Mnj}, \eqref{eq:Mnln}, and Equations~\eqref{eq:N00}, \eqref{eq:N0j}, \eqref{eq:N0l0}, \eqref{eq:Ni0}, \eqref{eq:Nij}, \eqref{eq:Nili}, \eqref{eq:Nn0}, \eqref{eq:Nnj}, \eqref{eq:Nnln} read
\begin{align*}
	&\mu^{(i,j)}_\alpha\Bigl(\bigl(y,x\bigr),\bigl(f_1,g_1\bigr),\ldots,\bigl(f_k,g_k\bigr),\bigl(p_1(x,x,y),p_1(x,y,y)\bigr),\ldots,\bigl(p_m(x,x,y),p_m(x,y,y)\bigr)\Bigr)\\
	=&\lambda^{(i,j)}_\alpha\Bigl(\bigl(y,x\bigr),\bigl(f_1,g_1\bigr),\ldots,\bigl(f_k,g_k\bigr),\bigl(p_1(x,x,y),p_1(x,y,y)\bigr),\ldots,\bigl(p_m(x,x,y),p_m(x,y,y)\bigr)\Bigr)\tag{Z($\alpha,\i,\j$)}\label{eq:Zaij}
\end{align*}
for all $i\in\{0,\ldots,n\}$, $j\in\{0,\ldots,\ell\}$ and $\alpha\in\{1,2\}$. Equations~\eqref{eq:Cij'} and \eqref{eq:Dij'} imply the equation
\begin{align*}
	\tau_{(i,j)}(&u,\vec{v},\vec{w},u',\vec{v}',\vec{w}',\lambda^{(i,j)}_1(u,\vec{v},\vec{w}),\lambda^{(i,j)}_2(u',\vec{v}',\vec{w}'))\\
	=\tau_{(i,j+1)}(&u,\vec{v},\vec{w},u',\vec{v}',\vec{w}',\mu^{(i,j+1)}_1(u,\vec{v},\vec{w}),\mu^{(i,j+1)}_2(u',\vec{v}',\vec{w}'))\tag{W($\i,\j$)}\label{eq:Wij}
\end{align*}
for all $i\in\{0,\ldots,n\}$ and $j\in\{0,\ldots,\ell-1\}$. Equation~\eqref{eq:Zaij} means more explicitly that
\begin{align*}
	\mu^{(i,j)}_\alpha(&y,f_1,\ldots,f_k,p_1(x,x,y),\ldots,p_m(x,x,y))\\
	=\lambda^{(i,j)}_\alpha(&y,f_1,\ldots,f_k,p_1(x,x,y),\ldots,p_m(x,x,y))\tag{Z($\alpha,\i,\j$)a}\label{eq:Zaija}
\end{align*}
and
\begin{align*}
	\mu^{(i,j)}_\alpha(&x,g_1,\ldots,g_k,p_1(x,y,y),\ldots,p_m(x,y,y))\\
	=\lambda^{(i,j)}_\alpha(&x,g_1,\ldots,g_k,p_1(x,y,y),\ldots,p_m(x,y,y))\tag{Z($\alpha,\i,\j$)b}\label{eq:Zaijb}
\end{align*}
for all $i\in\{0,\ldots,n\}$, $j\in\{0,\ldots,\ell\}$ and $\alpha\in\{1,2\}$. We display the Equations~\eqref{eq:B1'}, \eqref{eq:B2'}, \eqref{eq:Ti'}, \eqref{eq:Si'} and \eqref{eq:Wij} in the diagram
\begin{equation}\label{diagr:expandedtrailVW}
\begin{aligned}
	\xymatrixcolsep{7.4mm}
	\xymatrix{
		u \ar@{=}[rr]^-{\eqref{eq:B1'}}
		&& \tau_{(0,0)}[[\mu]]	\ar@{.}[r]
		&\tau_{(0,0)}[[\lambda]] \ar@{=}[rr]^-{\eqrefij{eq:Wij}{0}{0}}
		&&\tau_{(0,1)}[[\mu]]	\ar@{.}[r]
		& \tau_{(0,1)}[[\lambda]] \ar@{}[d]^-{\vdots}\\
		t_1'[[1,2]]\ar@{.}[d]
		&&\tau_{(0,\ell)}[[\lambda]] \ar@{=}[ll]_-{\eqrefi{eq:Ti'}{1}} \ar@{.}[r]
		& \tau_{(0,\ell)}[[\mu]] \ar@{=}[rr]^-{\eqrefij{eq:Wij}{0}{\ell-1}}
		&&\tau_{(0,\ell-1)}[[\lambda]] \ar@{.}[r]
		&\tau_{(0,\ell-1)}[[\mu]]\\
		t_1'[[2,1]]\ar@{=}[rr]^-{\eqrefi{eq:Si'}{1}}
		&& \tau_{(1,0)}[[\mu]] \ar@{.}[r]
		&\tau_{(1,0)}[[\lambda]] \ar@{=}[rr]^-{\eqrefij{eq:Wij}{1}{0}}
		&&\tau_{(1,1)}[[\mu]] \ar@{.}[r]
		& \tau_{(1,1)}[[\lambda]] \ar@{}[d]^-{\vdots}\\
		&&
		&
		&&
		& \ar@{}[d]^-{\vdots}\\
		t_n'[[1,2]]\ar@{.}[d]
		&&\tau_{(n-1,\ell)}[[\lambda]] \ar@{=}[ll]_-{\eqrefi{eq:Ti'}{n}} \ar@{.}[r]
		& \tau_{(n-1,\ell)}[[\mu]]\,\, \ar@{}[r]|(.5){}="A"
		&&\,\,\tau_{(n-1,\ell-1)}[[\lambda]] \ar@{}[l]|(.55){}="B" \ar@{=}"A";"B"^-{\eqrefij{eq:Wij}{n-1}{\ell-1}} \ar@{.}[r]
		&\tau_{(n-1,\ell-1)}[[\mu]]\\
		t_n'[[2,1]]\ar@{=}[rr]^-{\eqrefi{eq:Si'}{n}}
		&& \tau_{(n,0)}[[\mu]] \ar@{.}[r]
		&\tau_{(n,0)}[[\lambda]] \ar@{=}[rr]^-{\eqrefij{eq:Wij}{n}{0}}
		&&\tau_{(n,1)}[[\mu]]\ar@{.}[r]
		& \tau_{(n,1)}[[\lambda]] \ar@{}[d]^-{\vdots}\\
		u' \ar@{=}[rr]^-{\eqref{eq:B2'}}
		&& \tau_{(n,\ell)}[[\lambda]] \ar@{.}[r]
		&\tau_{(n,\ell)}[[\mu]] \ar@{=}[rr]^-{\eqrefij{eq:Wij}{n}{\ell-1}}
		&&\tau_{(n,\ell-1)}[[\lambda]] \ar@{.}[r]
		& \tau_{(n,\ell-1)}[[\mu]],
	}
\end{aligned}
\end{equation}
where
{\allowdisplaybreaks
\begin{align*}
	\tau_{(i,j)}[[\mu]]
	&:=
	\tau_{(i,j)}(u,\vec{v},\vec{w},u',\vec{v}',\vec{w}',\mu^{(i,j)}_1(u,\vec{v},\vec{w}),\mu^{(i,j)}_2(u',\vec{v}',\vec{w}')),
	\\
	\tau_{(i,j)}[[\lambda]]
	&:=
	\tau_{(i,j)}(u,\vec{v},\vec{w},u',\vec{v}',\vec{w}',\lambda^{(i,j)}_1(u,\vec{v},\vec{w}),\lambda^{(i,j)}_2(u',\vec{v}',\vec{w}')),
	\\
	t_i'[[1,2]]
	&:=
	t_i'(u,\vec{v},\vec{w},\vec{w},u',\vec{v}',\vec{w}',\vec{w}'),
	\\
	t_i'[[2,1]]
	&:=
	t_i'(u,\vec{v},\vec{w},\vec{w}',u',\vec{v}',\vec{w}',\vec{w}).
\end{align*}}%
In comparison to $t_i'[[1,2]]$, the second occurrences of $\vec{w}$ and~$\vec{w}'$, respectively, are swapped in $t_i'[[2,1]]$. Furthermore, for each $\tau_{(i,j)}$, the terms $\mu_1^{(i,j)},\mu_2^{(i,j)},\lambda_1^{(i,j)},\lambda_2^{(i,j)}$ satisfy the Equations~\eqref{eq:Zaija} and \eqref{eq:Zaijb}. We call an equation of form~(W) if it is of the form 
\begin{align*}
	&\omega(u,\vec{v},\vec{w},u',\vec{v}',\vec{w}',\gamma_1(u,\vec{v},\vec{w}),\gamma_2(u',\vec{v}',\vec{w}'))\\
	=&\chi(u,\vec{v},\vec{w},u',\vec{v}',\vec{w}',\delta_1(u,\vec{v},\vec{w}),\delta_2(u',\vec{v}',\vec{w}')),
\end{align*}
of form~(T) if it is of the form 
\begin{align*}
	&\xi(u,\vec{v},\vec{w},\vec{w},u',\vec{v}',\vec{w}',\vec{w}')\\
	=&\omega(u,\vec{v},\vec{w},u',\vec{v}',\vec{w}',\gamma_1(u,\vec{v},\vec{w}),\gamma_2(u',\vec{v}',\vec{w}')),
\end{align*}
and of form~(S) if it is of the form 
\begin{align*}
	&\xi(u,\vec{v},\vec{w},\vec{w}',u',\vec{v}',\vec{w}',\vec{w})\\
	=&\chi(u,\vec{v},\vec{w},u',\vec{v}',\vec{w}',\delta_1(u,\vec{v},\vec{w}),\delta_2(u',\vec{v}',\vec{w}'))
\end{align*}
where $\omega$, $\chi$, $\xi$ and $\gamma_1,\gamma_2,\delta_1,\delta_2$ are arbitrary terms with appropriate arity. In comparison to an equation of form~(T), the second occurrences of the variables $\vec{w}$ and $\vec{w}'$ in the first line, respectively, are swapped in an equation of form~(S). Then all Equations~\eqref{eq:Wij}, $i\in\{0,\ldots,n\}$ and $j\in\{0,\ldots,\ell-1\}$, are of form~(W), all Equations~\eqref{eq:Ti'}, $i\in\{1,\ldots,n\}$, are of form~(T), and all Equations~\eqref{eq:Si'}, $i\in\{1,\ldots,n\}$, are of form~(S), and we transform Diagram~\eqref{diagr:expandedtrailVW} for illustrative reasons into the following diagram:
\begin{footnotesize}
\begin{equation*}
	\begin{tikzcd}
		{} & {} & \cdots & {} & {} & {} & \cdots & {} & {} & {} & \cdots & {} & {}
		\arrow["{\eqref{eq:B1'}}", no head, from=1-1, to=1-2]
		\arrow["{(W)}", no head, from=1-2, to=1-3]
		\arrow["{(W)}", no head, from=1-3, to=1-4]
		\arrow["{(T)}", no head, from=1-4, to=1-5]
		\arrow["{(S)}", no head, from=1-5, to=1-6]
		\arrow["{(W)}", no head, from=1-6, to=1-7]
		\arrow["{(W)}", no head, from=1-7, to=1-8]
		\arrow["{(T)}", no head, from=1-8, to=1-9]
		\arrow["{(S)}", no head, from=1-9, to=1-10]
		\arrow["{(W)}", no head, from=1-10, to=1-11]
		\arrow["{(W)}", no head, from=1-11, to=1-12]
		\arrow["{\eqref{eq:B2'}}", no head, from=1-12, to=1-13]
	\end{tikzcd}
\end{equation*}
\end{footnotesize}%
We note that here the first line~\eqref{eq:B1'} encodes the equality of $u$ and $\tau_{(0,0)}[[\mu]]$, the first of the lines~$(W)$ encodes the equality of $\tau_{(0,0)}[[\lambda]]$ and $\tau_{(0,1)}[[\mu]]$, and so on. We show that we can replace each equation of form~(W) by an equation of form~(T) followed by an equation of form~(S) in the following sense. Then we get the diagram
\begin{scriptsize}
\begin{equation*}
	\begin{tikzcd}
		{} & {} && {} & {} && {} & {} & {} && {} & {} && {} & {},\\
		&& {} &&& {} &&&& {} &&& {} &&
		\arrow["{\eqref{eq:B1'}}", no head, from=1-1, to=1-2]
		\arrow["{(W)}", no head, from=1-2, to=1-4]
		\arrow[dotted, no head, from=1-4, to=1-5]
		\arrow["{(W)}", no head, from=1-5, to=1-7]
		\arrow["{(T)}", no head, from=1-7, to=1-8]
		\arrow["{(S)}", no head, from=1-8, to=1-9]
		\arrow["{(W)}", no head, from=1-9, to=1-11]
		\arrow[dotted, no head, from=1-11, to=1-12]
		\arrow["{(W)}", no head, from=1-12, to=1-14]
		\arrow["{\eqref{eq:B2'}}", no head, from=1-14, to=1-15]
		\arrow["{(T)}"', dashed, no head, from=1-2, to=2-3]
		\arrow["{(S)}"', dashed, no head, from=2-3, to=1-4]
		\arrow["{(T)}"', dashed, no head, from=1-5, to=2-6]
		\arrow["{(S)}"', dashed, no head, from=2-6, to=1-7]
		\arrow["{(T)}"', dashed, no head, from=1-9, to=2-10]
		\arrow["{(S)}"', dashed, no head, from=2-10, to=1-11]
		\arrow["{(T)}"', dashed, no head, from=1-12, to=2-13]
		\arrow["{(S)}"', dashed, no head, from=2-13, to=1-14]
	\end{tikzcd}
\end{equation*}
\end{scriptsize}%
where, following the dashed lines, equations of form~(T) and of form~(S) alternate. More precisely, we show that, for all $i\in\{0,\ldots,n\}$ and $j\in\{0,\ldots,\ell-1\}$, Equation~\eqref{eq:Wij} implies the existence of a $(2(k+2m+1))$-ary term $\pi_{(i,j)}$ such that
{\allowdisplaybreaks
\begin{align*}
	\tau_{(i,j)}&(u,\vec{v},\vec{w},u',\vec{v}',\vec{w}',\lambda_1^{(i,j)}(u,\vec{v},\vec{w}),\lambda_2^{(i,j)}(u',\vec{v}',\vec{w}'))\\*
	=\pi_{(i,j)}&(u,\vec{v},\vec{w},\vec{w},u',\vec{v}',\vec{w}',\vec{w}'),\tag{T($\i,\j$)'}\label{eq:Tij'}\\[1ex]
		\pi_{(i,j)}&(u,\vec{v},\vec{w},\vec{w}',u',\vec{v}',\vec{w}',\vec{w})\\*
		=\tau_{(i,j+1)}&(u,\vec{v},\vec{w},u',\vec{v}',\vec{w}',{\mu}_1^{(i,j+1)}(u,\vec{v},\vec{w}),{\mu}_2^{(i,j+1)}(u',\vec{v}',\vec{w}')).\tag{S($\i,\j$)'}\label{eq:Sij'}
\end{align*}}%
Equation~\eqref{eq:Tij'} is of form~(T) and Equation~\eqref{eq:Sij'} is of form~(S). One sees immediately that the choice
\begin{align*}
	\pi_{(i,j)}(u,\vec{v},\vec{w},\vec{\tilde{w}},u',\vec{v}',\vec{w}',\vec{\tilde{w}}) := \tau_{(i,j)}(u,\vec{v},\vec{w},u',\vec{v}',\vec{w}',\lambda_1^{(i,j)}(u,\vec{v},\vec{w}),\lambda_2^{(i,j)}(u',\vec{v}',\vec{w}'))
\end{align*}
fulfills both equations. By setting $N := n(\ell+1)+\ell$,
{\allowdisplaybreaks
\begin{align*}
	\sigma_{i(\ell+1)+j+1}
	&:= \tau_{(i,j)} \text{ for all } i\in\{0,\ldots,n\},\, j\in\{0,\ldots,\ell\},\\
	s_{i(\ell+1)+j+1}
	&:= \pi_{(i,j)} \text{ for all } i\in\{0,\ldots,n\},\, j\in\{0,\ldots,\ell-1\},\\
	s_{i(\ell+1)}
	&:= t'_i \text{ for all }  i \in \{1,\ldots,n\},\\
	\eta_\alpha^{(i(\ell+1)+j+1)}
	&:=\mu_\alpha^{(i,j)} \text{ for all } i\in\{0,\ldots,n\},\, j\in\{0,\ldots,\ell\}, \alpha \in \{1,2\},\\
	\epsilon_\alpha^{(i(\ell+1)+j+1)}
	&:=\lambda_\alpha^{(i,j)} \text{ for all }i\in\{0,\ldots,n\},\, j\in\{0,\ldots,\ell\}, \alpha \in \{1,2\},
\end{align*}}%
we get an integer $N\geqslant 1$, $(2(k+m+2))$-ary terms $\sigma_1,\ldots,\sigma_{N+1}$, $(2(k+2m+1))$-ary terms $s_1,\ldots,s_N$ and $(k+m+1)$-ary terms $\eta^{(i)}_1,\eta^{(i)}_2,\epsilon^{(i)}_1,\epsilon^{(i)}_2$, for all $i\in\{1,\ldots,N+1\}$, satisfying, by construction, all the desired equations.

Now we show that the existence of the terms as in the statement implies that $\mathbb{V}$ is a weakly Mal'tsev category. Equation~\eqref{eq:(f_i,g_i)inP} means that $(f_i(x,y),g_i(x,y))\in P$ for all $i\in\{1,\ldots,k\}$. Hence Equation~\eqref{eq:WMu} implies that
\begin{align*}
	&q_1(y,x)\\
	=\sigma_1\biggl(&q_1\bigl(y,x\bigr),q_1\bigl(f_1,g_1\bigr),\ldots,q_1\bigl(f_k,g_k\bigr),q_1\bigl(p_1(x,x,y),p_1(x,y,y)\bigr),\ldots,q_1\bigl(p_m(x,x,y),p_m(x,y,y)\bigr),\\
	& q_2\bigl(y,x\bigr),q_2\bigl(f_1,g_1\bigr),\ldots,q_2\bigl(f_k,g_k\bigr),q_2\bigl(p_1(x,x,y),p_1(x,y,y)\bigr),\ldots,q_2\bigl(p_m(x,x,y),p_m(x,y,y)\bigr),\\
	\eta_1^{(1)}\Bigl(&q_1\bigl(y,x\bigr),q_1\bigl(f_1,g_1\bigr),\ldots,q_1\bigl(f_k,g_k\bigr),q_1\bigl(p_1(x,x,y),p_1(x,y,y)\bigr),\ldots,q_1\bigl(p_m(x,x,y),p_m(x,y,y)\bigr)\Bigr),\\
	\eta_2^{(1)}\Bigl(&q_2\bigl(y,x\bigr),q_2\bigl(f_1,g_1\bigr),\ldots,q_2\bigl(f_k,g_k\bigr),q_2\bigl(p_1(x,x,y),p_1(x,y,y)\bigr),\ldots,q_2\bigl(p_m(x,x,y),p_m(x,y,y)\bigr)\Bigr)\biggr),
\end{align*}
where we write $(f_i,g_i)$ as shorter notation for $(f_i(x,y),g_i(x,y))$. Furthermore, we have, for all $i\in\{1,\ldots,N+1\}$ and $\alpha\in\{1,2\}$, that
{\allowdisplaybreaks
\begin{align*}
	&\eta_\alpha^{(i)}\Bigl(q_\alpha\bigl(y,x\bigr),q_\alpha\bigl(f_1,g_1\bigr),\ldots,q_\alpha\bigl(f_k,g_k\bigr),q_\alpha\bigl(p_1(x,x,y),p_1(x,y,y)\bigr),\ldots,q_\alpha\bigl(p_m(x,x,y),p_m(x,y,y)\bigr)\Bigr)\\*
	=&q_\alpha\biggl(\eta_\alpha^{(i)}\Bigl(\bigl(y,x\bigr),\bigl(f_1,g_1\bigr),\ldots,\bigl(f_k,g_k\bigr),\bigl(p_1(x,x,y),p_1(x,y,y)\bigr),\ldots,\bigl(p_m(x,x,y),p_m(x,y,y)\bigr)\Bigr)\biggr)\\
	=&q_\alpha\biggl(\epsilon_\alpha^{(i)}\Bigl(\bigl(y,x\bigr),\bigl(f_1,g_1\bigr),\ldots,\bigl(f_k,g_k\bigr),\bigl(p_1(x,x,y),p_1(x,y,y)\bigr),\ldots,\bigl(p_m(x,x,y),p_m(x,y,y)\bigr)\Bigr)\biggr)\\*
	=&\epsilon_\alpha^{(i)}\Bigl(q_\alpha\bigl(y,x\bigr),q_\alpha\bigl(f_1,g_1\bigr),\ldots,q_\alpha\bigl(f_k,g_k\bigr),q_\alpha\bigl(p_1(x,x,y),p_1(x,y,y)\bigr),\ldots,q_\alpha\bigl(p_m(x,x,y),p_m(x,y,y)\bigr)\Bigr),
\end{align*}}%
where we used Equations~\eqref{subeq:WMeta} and the definition of the algebra structure on $P\subseteq \mathsf{F}(x,y)\times \mathsf{F}(x,y)$. Thus, we have that
{\allowdisplaybreaks
\begin{align*}
	\sigma_i\biggl(&q_1\bigl(y,x\bigr),q_1\bigl(f_1,g_1\bigr),\ldots,q_1\bigl(f_k,g_k\bigr),q_1\bigl(p_1(x,x,y),p_1(x,y,y)\bigr),\ldots,q_1\bigl(p_m(x,x,y),p_m(x,y,y)\bigr),\\*
	&q_2\bigl(y,x\bigr),q_2\bigl(f_1,g_1\bigr),\ldots,q_2\bigl(f_k,g_k\bigr),q_2\bigl(p_1(x,x,y),p_1(x,y,y)\bigr),\ldots,q_2\bigl(p_m(x,x,y),p_m(x,y,y)\bigl),\\*
	\eta_1^{(i)}\Bigl(&q_1\bigl(y,x\bigr),q_1\bigl(f_1,g_1\bigr),\ldots,q_1\bigl(f_k,g_k\bigr),q_1\bigl(p_1(x,x,y),p_1(x,y,y)\bigr),\ldots,q_1\bigl(p_m(x,x,y),p_m(x,y,y)\bigr)\Bigr),\\*
	\eta_2^{(i)}\Bigl(&q_2\bigl(y,x\bigr),q_2\bigl(f_1,g_1\bigr),\ldots,q_2\bigl(f_k,g_k\bigr),q_2\bigl(p_1(x,x,y),p_1(x,y,y)\bigr),\ldots,q_2\bigl(p_m(x,x,y),p_m(x,y,y)\bigl)\Bigl)\biggl)\\
	= \sigma_i\biggl(&q_1\bigl(y,x\bigr),q_1\bigl(f_1,g_1\bigr),\ldots,q_1\bigl(f_k,g_k\bigr),q_1\bigl(p_1(x,x,y),p_1(x,y,y)\bigr),\ldots,q_1\bigl(p_m(x,x,y),p_m(x,y,y)\bigr),\\*
	&q_2\bigl(y,x\bigr),q_2\bigl(f_1,g_1\bigr),\ldots,q_2\bigl(f_k,g_k\bigr),q_2\bigl(p_1(x,x,y),p_1(x,y,y)\bigr),\ldots,q_2\bigl(p_m(x,x,y),p_m(x,y,y)\bigr),\\*
	\epsilon_1^{(i)}\Bigl(&q_1\bigl(y,x\bigr),q_1\bigl(f_1,g_1\bigr),\ldots,q_1\bigl(f_k,g_k\bigr),q_1\bigl(p_1(x,x,y),p_1(x,y,y)\bigr),\ldots,q_1\bigl(p_m(x,x,y),p_m(x,y,y)\bigr)\Bigr),\\*
	\epsilon_2^{(i)}\Bigl(&q_2\bigl(y,x\bigr),q_2\bigl(f_1,g_1\bigr),\ldots,q_2\bigl(f_k,g_k\bigr),q_2\bigl(p_1(x,x,y),p_1(x,y,y)\bigr),\ldots,q_2\bigl(p_m(x,x,y),p_m(x,y,y)\bigr)\Bigr)\biggr)
\end{align*}}%
for all $i\in\{1,\ldots,N+1\}$. Next we have that
{\allowdisplaybreaks
\begin{align*}
	\sigma_i\biggl(&q_1\bigl(y,x\bigr),q_1\bigl(f_1,g_1\bigr),\ldots,q_1\bigl(f_k,g_k\bigr),q_1\bigl(p_1(x,x,y),p_1(x,y,y)\bigr),\ldots,q_1\bigl(p_m(x,x,y),p_m(x,y,y)\bigr),\\*
	&q_2\bigl(y,x\bigr),q_2\bigl(f_1,g_1\bigr),\ldots,q_2\bigl(f_k,g_k\bigr),q_2\bigl(p_1(x,x,y),p_1(x,y,y)\bigr),\ldots,q_2\bigl(p_m(x,x,y),p_m(x,y,y)\bigr),\\*
	\epsilon_1^{(i)}\Bigl(&q_1\bigl(y,x\bigr),q_1\bigl(f_1,g_1\bigr),\ldots,q_1\bigl(f_k,g_k\bigr),q_1\bigl(p_1(x,x,y),p_1(x,y,y)\bigr),\ldots,q_1\bigl(p_m(x,x,y),p_m(x,y,y)\bigr)\Bigr),\\*
	\epsilon_2^{(i)}\Bigl(&q_2\bigl(y,x\bigr),q_2\bigl(f_1,g_1\bigr),\ldots,q_2\bigl(f_k,g_k\bigr),q_2\bigl(p_1(x,x,y),p_1(x,y,y)\bigr),\ldots,q_2\bigl(p_m(x,x,y),p_m(x,y,y)\bigr)\Bigr)\biggr)\\
	=s_i\Bigl(
	&q_1\bigl(y,x\bigr), q_1\bigl(f_1,g_1\bigr),\ldots,q_1\bigl(f_k,g_k\bigr), q_1\bigl(p_1(x,x,y),p_1(x,y,y)\bigr),\ldots,q_1\bigl(p_m(x,x,y),p_m(x,y,y)\bigr),\\
	&q_1\bigl(p_1(x,x,y),p_1(x,y,y)\bigr),\ldots,q_1\bigl(p_m(x,x,y),p_m(x,y,y)\bigr),\\
	&q_2\bigl(y,x\bigr), q_2\bigl(f_1,g_1\bigr),\ldots,q_2\bigl(f_k,g_k\bigr), q_2\bigl(p_1(x,x,y),p_1(x,y,y)\bigr),\ldots,q_2\bigl(p_m(x,x,y),p_m(x,y,y)\bigr),\\
	&q_2\bigl(p_1(x,x,y),p_1(x,y,y)\bigr),\ldots,q_2\bigl(p_m(x,x,y),p_m(x,y,y)\bigr)
	\Bigr)\\
	=s_i\Bigl(
	&q_1\bigl(y,x\bigr), q_1\bigl(f_1,g_1\bigr),\ldots,q_1\bigl(f_k,g_k\bigr), q_1\bigl(p_1(x,x,y),p_1(x,y,y)\bigr),\ldots,q_1\bigl(p_m(x,x,y),p_m(x,y,y)\bigr),\\
	&q_2\bigl(p_1(x,x,y),p_1(x,y,y)\bigr),\ldots,q_2\bigl(p_m(x,x,y),p_m(x,y,y)\bigr),\\
	&q_2\bigl(y,x\bigr), q_2\bigl(f_1,g_1\bigr),\ldots,q_2\bigl(f_k,g_k\bigr), q_2\bigl(p_1(x,x,y),p_1(x,y,y)\bigr),\ldots,q_2\bigl(p_m(x,x,y),p_m(x,y,y)\bigr),\\
	&q_1\bigl(p_1(x,x,y),p_1(x,y,y)\bigr),\ldots,q_1\bigl(p_m(x,x,y),p_m(x,y,y)\bigr)
	\Bigr)\\
	=\sigma_{i+1}\biggl(&q_1\bigl(y,x\bigr),q_1\bigl(f_1,g_1\bigr),\ldots,q_1\bigl(f_k,g_k\bigr),q_1\bigl(p_1(x,x,y),p_1(x,y,y)\bigr),\ldots,q_1\bigl(p_m(x,x,y),p_m(x,y,y)\bigr),\\*
	&q_2\bigl(y,x\bigr),q_2\bigl(f_1,g_1\bigr),\ldots,q_2\bigl(f_k,g_k\bigr),q_2\bigl(p_1(x,x,y),p_1(x,y,y)\bigr),\ldots,q_2\bigl(p_m(x,x,y),p_m(x,y,y)\bigr),\\*
	\eta_1^{(i+1)}\Bigl(&q_1\bigl(y,x\bigr),q_1\bigl(f_1,g_1\bigr),\ldots,q_1\bigl(f_k,g_k\bigr),q_1\bigl(p_1(x,x,y),p_1(x,y,y)\bigr),\ldots,q_1\bigl(p_m(x,x,y),p_m(x,y,y)\bigr)\Bigr),\\*
	\eta_2^{(i+1)}\Bigl(&q_2\bigl(y,x\bigr),q_2\bigl(f_1,g_1\bigr),\ldots,q_2\bigl(f_k,g_k\bigr),q_2\bigl(p_1(x,x,y),p_1(x,y,y)\bigr),\ldots,q_2\bigl(p_m(x,x,y),p_m(x,y,y)\bigr)\Bigr)\Biggr)
\end{align*}}%
for all $i\in\{1,\ldots,N\}$, where we applied Equation~\eqref{eq:WModd} in the first step. In the second step we used that
\begin{equation}\label{eq:q1q2pi}
	q_1(p_i(x,x,y),p_i(x,y,y))=q_2(p_i(x,x,y),p_i(x,y,y))
\end{equation}
for all $i\in\{1,\ldots,m\}$, which is true by Equation~\eqref{eq:im([e_1,e_2])} and the definition of $(q_1,q_2)$ as the cokernel pair of $[e_1,e_2]$. In the third step we utilized Equation~\eqref{eq:WMeven}. Equation~\eqref{eq:WMu'} implies that
\begin{align*}
	\sigma_{N+1}\biggl(&q_1\bigl(y,x\bigr),q_1\bigl(f_1,g_1\bigr),\ldots,q_1\bigl(f_k,g_k\bigr),q_1\bigl(p_1(x,x,y),p_1(x,y,y)\bigr),\ldots,q_1\bigl(p_m(x,x,y),p_m(x,y,y)\bigr),\\
	&q_2\bigl(y,x\bigr),q_2\bigl(f_1,g_1\bigr),\ldots,q_2\bigl(f_k,g_k\bigr),q_2\bigl(p_1(x,x,y),p_1(x,y,y)\bigr),\ldots,q_2\bigl(p_m(x,x,y),p_m(x,y,y)\bigr),\\
	\epsilon_1^{(N+1)}\Bigl(&q_1\bigl(y,x\bigr),q_1\bigl(f_1,g_1\bigr),\ldots,q_1\bigl(f_k,g_k\bigr),q_1\bigl(p_1(x,x,y),p_1(x,y,y)\bigr),\ldots,q_1\bigl(p_m(x,x,y),p_m(x,y,y)\bigr)\Bigr),\\
	\epsilon_2^{(N+1)}\Bigl(&q_2\bigl(y,x\bigr),q_2\bigl(f_1,g_1\bigr),\ldots,q_2\bigl(f_k,g_k\bigr),q_2\bigl(p_1(x,x,y),p_1(x,y,y)\bigr),\ldots,q_2\bigl(p_m(x,x,y),p_m(x,y,y)\bigr)\Bigr)\biggr)\\
	=&q_2(y,x).
\end{align*} 
Combining the above identities, we get that $q_1(y,x)=q_2(y,x)$ which means that $\mathbb{V}$ is a weakly Mal'tsev category by Lemma~\ref{lem:WMif and only ifprojection}.
\end{proof}

\begin{example}
We consider the special case $N=0$ in Theorem~\ref{thm:syntaxWM}, i.e., we suppose that the variety $\mathbb{V}$ admits binary terms $f_1,g_1,\ldots,f_k,g_k$, ternary terms $p_1,\ldots,p_m$, where $k,m\geqslant 0$ are integers, a $(2(k+m+2))$-ary term $\sigma_1$ and $(k+m+1)$-ary terms $\eta^{(1)}_1,\eta^{(1)}_2,\epsilon^{(1)}_1,\epsilon^{(1)}_2$ that fulfill all the required conditions. Using Equations~\eqref{subeq:WMuu'}, Equation~\eqref{eq:WMetay} for $\eta^{(1)}_1$ and $\epsilon^{(1)}_1$, and Equation~\eqref{eq:WMetax} for $\eta^{(1)}_2$ and $\epsilon^{(1)}_2$, we get that the identity
{\allowdisplaybreaks
\begin{align*}
	y =&\sigma_1(y,f_1,\ldots,f_k,p_1(x,x,y),\ldots,p_m(x,x,y),x,g_1,\ldots,g_k,p_1(x,y,y),\ldots,p_m(x,y,y),\\*
	&\phantom{\sigma_1(}\eta^{(1)}_1(y,f_1,\ldots,f_k,p_1(x,x,y),\ldots,p_m(x,x,y)),\\*
	&\phantom{\sigma_1(}\eta^{(1)}_2(x,g_1,\ldots,g_k,p_1(x,y,y),\ldots,p_m(x,y,y)))\\
	=&\sigma_1(y,f_1,\ldots,f_k,p_1(x,x,y),\ldots,p_m(x,x,y),x,g_1,\ldots,g_k,p_1(x,y,y),\ldots,p_m(x,y,y),\\*
	&\phantom{\sigma_1(}\epsilon^{(1)}_1(y,f_1,\ldots,f_k,p_1(x,x,y),\ldots,p_m(x,x,y)),\\*
	&\phantom{\sigma_1(}\epsilon^{(1)}_2(x,g_1,\ldots,g_k,p_1(x,y,y),\ldots,p_m(x,y,y)))\\*
	=&x
\end{align*}}%
holds in~$\mathbb{V}$. In particular, $\mathbb{V}$ is a Mal'tsev variety with Mal'tsev term $p(x,y,z):= x$.
\end{example}

\begin{example}[Distributive lattices]\label{ex:distributivelattices}
In~\cite{martins-ferreira:2012}, it is shown that a variety $\mathbb{L}$ of lattices is a weakly Mal'tsev category if and only if it is a variety of distributive lattices. For such a variety $\mathbb{L}$, we want to find integers $k,m,N\geqslant 0$ and terms $f_1,g_1,\ldots,f_k,g_k$, $p_1,\ldots,p_m$, $ s_1,\ldots,s_N $, $\sigma_1,\ldots,\sigma_{N+1}$ and $\eta^{(i)}_1,\eta^{(i)}_2,\epsilon^{(i)}_1,\epsilon^{(i)}_2$ for all $i\in\{1,\ldots,N+1\}$ that fulfill the conditions of Theorem~\ref{thm:syntaxWM}. We observe that we can perform the following calculation in $\mathbb{L}$ to show that $q_1(y,x)=q_2(y,x)$:
{\allowdisplaybreaks
\begin{align*}
	&q_1(y,x)\\*
	=&q_1(y,x)\land (q_1(y,x) \lor q_1(x,y))\tag{absorption law}\\
	=&q_1(y,x)\land(q_1(y,y)\lor q_1(x,x))\tag{$q_1$ morphism and commutativity of $\lor$}\\
	=&q_1(y,x)\land(q_2(y,y)\lor q_2(x,x))\tag{$q_1(y,y)=q_2(y,y)$ and $q_1(x,x)=q_2(x,x)$}\\
	=&q_1(y,x)\land(q_2(y,x)\lor q_2(x,y))\tag{$q_2$ morphism and commutativity of $\lor$}\\
	=&q_1(y,x)\land(q_2(y,x)\lor q_1(x,y))\tag{$q_2(x,y)=q_1(x,y)$}\\
	=&(q_1(y,x)\land q_2(y,x))\lor (q_1(y,x)\land q_1(x,y))\tag{distributivity of $\land$ over $\lor$}\\
	=&(q_1(y,x)\land q_2(y,x))\lor (q_1(y,y)\land q_1(x,x))\tag{$q_1$ morphism and commutativity of $\land$}\\
	=&(q_2(y,x)\land q_1(y,x))\lor (q_1(y,y)\land q_1(x,x))\tag{commutativity of $\land$}\\
	=&(q_2(y,x)\land q_1(y,x))\lor (q_2(y,y)\land q_2(x,x))\tag{$q_1(y,y)=q_2(y,y)$ and $q_1(x,x)=q_2(x,x)$}\\
	=&(q_2(y,x)\land q_1(y,x))\lor (q_2(y,x)\land q_2(x,y))&\tag{$q_2$ morphism and commutativity of $\land$}\\
	=&q_2(y,x)\land (q_1(y,x)\lor q_2(x,y))\tag{distributivity of $\land$ over $\lor$}\\
	=&q_2(y,x)\land (q_1(y,x)\lor q_1(x,y))\tag{$q_2(x,y)=q_1(x,y)$}\\
	=&q_2(y,x)\land (q_1(y,y)\lor q_1(x,x))\tag{$q_1$ morphism and commutativity of $\lor$}\\
	=&q_2(y,x)\land (q_2(y,y)\lor q_2(x,x))\tag{$q_1(y,y)=q_2(y,y)$ and $q_1(x,x)=q_2(x,x)$}\\
	=&q_2(y,x)\land (q_2(y,x)\lor q_2(x,y))\tag{$q_2$ morphism and commutativity of $\lor$}\\*
	=&q_2(y,x)\tag{absorption law}.
\end{align*}}%
In the proof of Theorem~\ref{thm:syntaxWM}, we showed, given terms satisfying the conditions as in the statement, that $q_1(y,x)=q_2(y,x)$ in the following way:
{\allowdisplaybreaks
\begin{align*}
	&q_1(y,x)\\*
	=&\sigma_1(\eta)
	\tag{Equation~\eqref{eq:WMu}}\\
	=&\sigma_1(\epsilon)
	\tag{(M) and Equations~\eqref{subeq:WMeta} for $i=1$}\\	
	=&s_1(1,2)
	\tag{Equation~\eqref{eq:WModd} for $i=1$}\\
	=&s_1(2,1)
	\tag{Equation~\eqref{eq:q1q2pi}}\\
	=&\sigma_2(\eta)
	\tag{Equation~\eqref{eq:WMeven} for $i=1$}\\
	=&\sigma_2(\epsilon)
	\tag{(M) and Equations~\eqref{subeq:WMeta} for $i=2$}\\
	=&s_2(1,2)
	\tag{Equation~\eqref{eq:WModd} for $i=2$}\\
	=&s_2(2,1)
	\tag{Equation~\eqref{eq:q1q2pi}}\\	
	=&\sigma_3(\eta)
	\tag{Equation~\eqref{eq:WMeven} for $i=2$}\\*
	\vdots\\*
	=&\sigma_{N+1}(\epsilon)
	\tag{(M) and Equations~\eqref{subeq:WMeta} for $i=N+1$}\\*
	=&q_2(y,x),\tag{Equation~\eqref{eq:WMu'}}
\end{align*}}%
where we set
{\allowdisplaybreaks
\begin{align*}
	\sigma_i(\eta):= \sigma_i\Bigl(
	&q_1(y,x),q_1(f_1,g_1),\ldots,q_1(p_1(x,x,y),p_1(x,y,y)),\ldots,\\*
	&q_2(y,x),q_2(f_1,g_1),\ldots,q_2(p_1(x,x,y),p_1(x,y,y)),\ldots,\\*
	\eta^{(i)}_1\bigl(&q_1(y,x),q_1(f_1,g_1),\ldots,q_1(p_1(x,x,y),p_1(x,y,y)),\ldots\bigr),\\*
	\eta^{(i)}_2\bigl(&q_2(y,x),q_2(f_1,g_1),\ldots,q_2(p_1(x,x,y),p_1(x,y,y)),\ldots\bigr)\Bigr),\\
	\sigma_i(\epsilon):= \sigma_i\Bigl(
	&q_1(y,x),q_1(f_1,g_1),\ldots,q_1(p_1(x,x,y),p_1(x,y,y)),\ldots,\\*
	&q_2(y,x),q_2(f_1,g_1),\ldots,q_2(p_1(x,x,y),p_1(x,y,y)),\ldots,\\*
	\epsilon^{(i)}_1\bigl(&q_1(y,x),q_1(f_1,g_1),\ldots,q_1(p_1(x,x,y),p_1(x,y,y)),\ldots\bigr),\\*
	\epsilon^{(i)}_2\bigl(&q_2(y,x),q_2(f_1,g_1),\ldots,q_2(p_1(x,x,y),p_1(x,y,y)),\ldots\bigr)\Bigr),\\
	s_j(1,2):= s_i\Bigl(
		&q_1(y,x),q_1(f_1,g_1),\ldots,\\*
		&q_1(p_1(x,x,y),p_1(x,y,y)),\ldots,\\*
		&q_1(p_1(x,x,y),p_1(x,y,y)),\ldots,\\*
		&q_2(y,x),q_2(f_1,g_1),\ldots,\\*
		&q_2(p_1(x,x,y),p_1(x,y,y)),\ldots,\\*
		&q_2(p_1(x,x,y),p_1(x,y,y)),\ldots\Bigr),\\
	s_j(2,1):= s_i\Bigl(
		&q_1(y,x),q_1(f_1,g_1),\ldots,\\*
		&q_1(p_1(x,x,y),p_1(x,y,y)),\ldots,\\*
		&q_2(p_1(x,x,y),p_1(x,y,y)),\ldots,\\*
		&q_2(y,x),q_2(f_1,g_1),\ldots,\\*
		&q_2(p_1(x,x,y),p_1(x,y,y)),\ldots,\\*
		&q_1(p_1(x,x,y),p_1(x,y,y)),\ldots\Bigr)
\end{align*}}%
for all $i\in\{1,\ldots,N+1\}$ and all $j\in\{1,\ldots,N\}$ and (M) means that we used that $q_1$ and $q_2$ are morphisms. Hence our ansatz is that the first step, where we use the absorption law, should correspond to the step where we use Equation~\eqref{eq:WMu}. And the last step, where we use again the absorption law, should correspond to the step where we use Equation~\eqref{eq:WMu'}. A step in our concrete calculation for distributive lattices where we use an equation should either correspond to a step from $\sigma_i(\epsilon)$ to $s_i(1,2)$ or from $s_i(2,1)$ to $\sigma_{i+1}(\eta)$ for some~$i$. A step where we use that $q_1$ or $q_2$ is a morphism and an equation should correspond to a step from $\sigma_i(\eta)$ to $\sigma_i(\epsilon)$. A step where we use that $q_1(y,y)=q_2(y,y)$, $q_1(x,x)=q_2(x,x)$ or $q_1(x,y)=q_2(x,y)$ should correspond to a step from $s_i(1,2)$ to $s_i(2,1)$. For $N=5$ we found integers $k,m\geq 0$ and terms $f_1,g_1,\ldots,f_k,g_k$, $p_1,\ldots,p_m$, $s_1,\ldots,s_5$, $\sigma_1,\ldots,\sigma_6$ and $\eta^{(i)}_1,\eta^{(i)}_2,\epsilon^{(i)}_1,\epsilon^{(i)}_2$ for all $i\in\{1,\ldots,6\}$ such that
{\allowdisplaybreaks
\begin{align*}
	q_1(y,x)\land (q_1(y,x) \lor q_1(x,y))
	&=\sigma_1(\eta),\\*
	q_1(y,x)\land(q_1(y,y)\lor q_1(x,x))
	&=\sigma_1(\epsilon)=s_1(1,2),\\
	q_1(y,x)\land(q_2(y,y)\lor q_2(x,x))
	&=s_1(2,1)=\sigma_2(\eta),\\
	q_1(y,x)\land(q_2(y,x)\lor q_2(x,y))
&=\sigma_2(\epsilon)=s_2(1,2),\\
	q_1(y,x)\land(q_2(y,x)\lor q_1(x,y))
	&=s_2(2,1),\\
	(q_1(y,x)\land q_2(y,x))\lor (q_1(y,x)\land q_1(x,y))
	&=\sigma_3(\eta),\\
	(q_1(y,x)\land q_2(y,x))\lor (q_1(y,y)\land q_1(x,x))	&=\sigma_3(\epsilon),\\
	(q_2(y,x)\land q_1(y,x))\lor (q_1(y,y)\land q_1(x,x))
	&=s_3(1,2),\\
	(q_2(y,x)\land q_1(y,x))\lor (q_2(y,y)\land q_2(x,x))
	&=s_3(2,1)=\sigma_4(\eta),\\
	(q_2(y,x)\land q_1(y,x))\lor (q_2(y,x)\land q_2(x,y))
	&=\sigma_4(\epsilon),\\
	q_2(y,x)\land (q_1(y,x)\lor q_2(x,y))
	&=s_4(1,2),\\
	q_2(y,x)\land (q_1(y,x)\lor q_1(x,y))
	&=s_4(2,1)=\sigma_5(\eta),\\
	q_2(y,x)\land (q_1(y,y)\lor q_1(x,x))
	&=\sigma_5(\epsilon)=s_5(1,2),\\
	q_2(y,x)\land (q_2(y,y)\lor q_2(x,x))
	&=s_5(2,1)=\sigma_6(\eta),\\*
	q_2(y,x)\land (q_2(y,x)\lor q_2(x,y))
	&=\sigma_6(\epsilon),
\end{align*}}%
and the conditions of Theorem~\ref{thm:syntaxWM} are satisfied. One can check that $k:=0$, $m:=3$ and the following terms fulfill the requirements:
\begin{align*}
	p_1(x,y,z)&:= x, &p_2(x,y,z)&:= y, &p_3(x,y,z)&:= z,
\end{align*}
{\allowdisplaybreaks
\begin{align*}
	s_1(u,w_1,w_2,w_3,\tilde{w}_1,\tilde{w}_2,\tilde{w}_3,u',w'_1,w'_2,w'_3,\tilde{w}'_1,\tilde{w}'_2,\tilde{w}'_3)&:=u\land (\tilde{w}_3 \lor \tilde{w}_1),\\*
	s_2(u,w_1,w_2,w_3,\tilde{w}_1,\tilde{w}_2,\tilde{w}_3,u',w'_1,w'_2,w'_3,\tilde{w}'_1,\tilde{w}'_2,\tilde{w}'_3)&:=u\land (u' \lor \tilde{w}'_2),\\
	s_3(u,w_1,w_2,w_3,\tilde{w}_1,\tilde{w}_2,\tilde{w}_3,u',w'_1,w'_2,w'_3,\tilde{w}'_1,\tilde{w}'_2,\tilde{w}'_3)&:=(u' \land u) \lor (\tilde{w}_3 \land \tilde{w}_1),\\
	s_4(u,w_1,w_2,w_3,\tilde{w}_1,\tilde{w}_2,\tilde{w}_3,u',w'_1,w'_2,w'_3,\tilde{w}'_1,\tilde{w}'_2,\tilde{w}'_3)&:=u'\land (u \lor \tilde{w}'_2),\\*
	s_5(u,w_1,w_2,w_3,\tilde{w}_1,\tilde{w}_2,\tilde{w}_3,u',w'_1,w'_2,w'_3,\tilde{w}'_1,\tilde{w}'_2,\tilde{w}'_3)&:=u'\land (\tilde{w}_3 \lor \tilde{w}_1),
\end{align*}}%
{\allowdisplaybreaks
\begin{align*}
	\sigma_1(u,w_1,w_2,w_3,u',w'_1,w'_2,w'_3,a,b)&:= u \land a,\\*
	\sigma_2(u,w_1,w_2,w_3,u',w'_1,w'_2,w'_3,a,b)&:= u \land b,\\
	\sigma_3(u,w_1,w_2,w_3,u',w'_1,w'_2,w'_3,a,b)&:= (u \land u') \lor a,\\
	\sigma_4(u,w_1,w_2,w_3,u',w'_1,w'_2,w'_3,a,b)&:= (u' \land u) \lor b,\\
	\sigma_5(u,w_1,w_2,w_3,u',w'_1,w'_2,w'_3,a,b)&:= u' \land a,\\*
	\sigma_6(u,w_1,w_2,w_3,u',w'_1,w'_2,w'_3,a,b)&:= u' \land b,
\end{align*}}%
{\allowdisplaybreaks
\begin{align*}
	\eta^{(1)}_1(u,w_1,w_2,w_3)&:= u \lor w_2,
	&\epsilon^{(1)}_1(u,w_1,w_2,w_3)&:= w_3 \lor w_1,\\*
	\eta^{(1)}_2(u,w_1,w_2,w_3)&:= u,
	&\epsilon^{(1)}_2(u,w_1,w_2,w_3)&:= u,\\
	\eta^{(2)}_1(u,w_1,w_2,w_3)&:= u,
	&\epsilon^{(2)}_1(u,w_1,w_2,w_3)&:= u,\\
	\eta^{(2)}_2(u,w_1,w_2,w_3)&:= w_3 \lor w_1,
	&\epsilon^{(2)}_2(u,w_1,w_2,w_3)&:= u \lor w_2,\\
	\eta^{(3)}_1(u,w_1,w_2,w_3)&:= u \land w_2,
	&\epsilon^{(3)}_1(u,w_1,w_2,w_3)&:= w_3 \land w_1,\\
	\eta^{(3)}_2(u,w_1,w_2,w_3)&:= u,
	&\epsilon^{(3)}_2(u,w_1,w_2,w_3)&:= u,\\
	\eta^{(4)}_1(u,w_1,w_2,w_3)&:= u,
	&\epsilon^{(4)}_1(u,w_1,w_2,w_3)&:= u,\\
	\eta^{(4)}_2(u,w_1,w_2,w_3)&:= w_3 \land w_1,
	&\epsilon^{(4)}_2(u,w_1,w_2,w_3)&:= u \land w_2,\\
	\eta^{(5)}_1(u,w_1,w_2,w_3)&:= u \lor w_2,
	&\epsilon^{(5)}_1(u,w_1,w_2,w_3)&:= w_3 \lor w_1,\\
	\eta^{(5)}_2(u,w_1,w_2,w_3)&:= u,
	&\epsilon^{(5)}_2(u,w_1,w_2,w_3)&:= u,\\
	\eta^{(6)}_1(u,w_1,w_2,w_3)&:= u,
	&\epsilon^{(6)}_1(u,w_1,w_2,w_3)&:= u,\\*
	\eta^{(6)}_2(u,w_1,w_2,w_3)&:= w_3 \lor w_1,
	&\epsilon^{(6)}_2(u,w_1,w_2,w_3)&:= u \lor w_2.
\end{align*}}%
\end{example}

\section{Another Mal'tsev-like property}\label{sect:AnotherMal'tsev-likeproperty}

In this section, we discuss the observation that the condition which is obtained from the syntactic characterization of weakly Mal'tsev varieties from Theorem~\ref{thm:syntaxWM} by dropping Equation~\eqref{eq:(f_i,g_i)inP} characterizes the varieties in which any reflexive regular relation is an equivalence relation. A \emph{regular relation}, in the context of finitely complete categories, is a monomorphism $r\colon R\rightarrowtail X\times Y$ which is \emph{regular}, i.e., an equalizer of a parallel pair of morphisms $f,g\colon X\times Y\to Z$.

We call a class $\mathcal{M}$ of binary relations in a finitely complete category $\mathfrak{C}$ \emph{pullback-stable} if, for any relation $r\colon R\rightarrowtail X\times Y$ in $\mathcal{M}$ and any morphisms $f\colon A\to X$ and $g\colon B\to Y$, the relation $\rho\colon P\rightarrowtail A\times B$ given by any pullback
\begin{equation*}
	\begin{tikzcd}
		P & {A\times B} \\
		R & {X\times Y,}
		\arrow["{f\times g}", from=1-2, to=2-2]			\arrow["{r}"', tail, from=2-1, to=2-2]
		\arrow["{p_R}"', from=1-1, to=2-1]
		\arrow["{\rho}", tail, from=1-1, to=1-2]
		\arrow["\scalebox{2}{$\lrcorner$}"{anchor=center, pos=0.125}, draw=none, from=1-1, to=2-2]
	\end{tikzcd}
\end{equation*}
lies also in~$\mathcal{M}$. The classes $\mathcal{M}_\text{all}$ of all binary relations in~$\mathfrak{C}$, $\mathcal{M}_{\text{strong}}$ of strong binary relations and $\mathcal{M}_{\text{reg}}$ of regular binary relations are pullback-stable.

\begin{definition}[$\mathcal{M}$-Mal'tsev category]
Let $\mathfrak{C}$ be a finitely complete category and $\mathcal{M}$ be a pullback-stable class of binary relations in~$\mathfrak{C}$. We call $\mathfrak{C}$ an \emph{$\mathcal{M}$-Mal'tsev category} if any reflexive relation $r\colon R\rightarrowtail X\times X$ in $\mathcal{M}$ is an equivalence relation.
\end{definition}

Let $\mathfrak{C}$ be a finitely complete category. By Definition~\ref{def:Mal'tsevcategory}, it is $\mathcal{M}_{\text{all}}$-Mal'tsev if and only if it is a Mal'tsev category. Furthermore, by Proposition~\ref{prop:characterizationweaklyMal'tsevcategories}, it is $\mathcal{M}_{\text{strong}}$-Mal'tsev if and only if it is a weakly Mal'tsev category. In this section, we consider $\mathcal{M}_{\text{reg}}$-Mal'tsev categories. 

The following proposition recovers Proposition~\ref{prop:characterizationMal'tsevcategoriesrelations} and parts of Proposition~\ref{prop:characterizationweaklyMal'tsevcategories}. In~\cite{janelidze.martins-ferreira:2012}, the authors show similarly the following equivalences for $\mathcal{M}$ being a proper class of binary relations in a category with pullbacks and equalizers (see also~\cite{martins-ferreira:2020}, Theorem~1).

\begin{proposition}
Let $\mathfrak{C}$ be a finitely complete category and $\mathcal{M}$ be a pullback-stable class of binary relations in~$\mathfrak{C}$. Then the following conditions are equivalent:
\begin{enumerate}[label=\arabic*. , ref=\arabic*.]
	\item\label{prop:M-Maltsev category:def} The category $\mathfrak{C}$ is an $\mathcal{M}$-Mal'tsev category.
	\item\label{prop:M-Maltsev category:symmetric} Any reflexive relation in $\mathcal{M}$ is symmetric.
	\item\label{prop:M-Maltsev category:transitive} Any reflexive relation in $\mathcal{M}$ is transitive.
	\item\label{prop:M-Maltsev category:difunctional} Any relation in $\mathcal{M}$ is difunctional.
\end{enumerate}
\end{proposition}

\begin{proof}
The implications "\ref{prop:M-Maltsev category:def}~$\Rightarrow$~\ref{prop:M-Maltsev category:symmetric}" and "\ref{prop:M-Maltsev category:def}~$\Rightarrow$~\ref{prop:M-Maltsev category:transitive}" are trivial. The implication "\ref{prop:M-Maltsev category:difunctional}~$\Rightarrow$~\ref{prop:M-Maltsev category:def}" holds since, as it is well-known, any reflexive relation is an equivalence relation if and only if it is difunctional. For the implication "\ref{prop:M-Maltsev category:symmetric}~$\Rightarrow$~\ref{prop:M-Maltsev category:difunctional}", we consider a relation $r=(r_1,r_2)\colon R\rightarrowtail X\times Y$ in~$\mathcal{M}$. Since $\mathcal{M}$ is assumed to be pullback-stable, the relation $\rho\colon P\rightarrowtail R\times R$ in the pullback
\begin{equation*}
	\begin{tikzcd}
		P & {R\times R} \\
		R & {X\times Y,}
		\arrow["{r_1\times r_2}", from=1-2, to=2-2]
		\arrow["{r}"', tail, from=2-1, to=2-2]
		\arrow["{p_R}"', from=1-1, to=2-1]
		\arrow["{\rho}", tail, from=1-1, to=1-2]
		\arrow["\scalebox{2}{$\lrcorner$}"{anchor=center, pos=0.125}, draw=none, from=1-1, to=2-2]
	\end{tikzcd}
\end{equation*}
lies also in~$\mathcal{M}$. Furthermore, it is clear that it is reflexive. Hence, by assumption, $\rho$ is symmetric. It is a classical result that $\rho$ being symmetric is equivalent to $r$ being difunctional. For the implication "\ref{prop:M-Maltsev category:transitive}~$\Rightarrow$~\ref{prop:M-Maltsev category:difunctional}", we note that $r$ is difunctional if and only if $\rho$ is transitive.
\end{proof}

As in the preceding section, we will adapt a proof of the classical theorem of Mal'tsev to the case of $\mathcal{M}_{\text{reg}}$-Mal'tsev varieties. This time, we use the characterization of Mal'tsev categories as those finitely complete categories in which every reflexive relation is symmetric as recalled in Proposition~\ref{prop:characterizationMal'tsevcategoriesrelations}. Let
\begin{equation*}
	\begin{tikzcd}
		R && {\mathsf{F}(x,y)\times\mathsf{F}(x,y)}
		\arrow["r", tail, from=1-1, to=1-3]
	\end{tikzcd}
\end{equation*}
be the smallest relation on $\mathsf{F}(x,y)$ which contains $(x,x)$, $(x,y)$ and $(y,y)$. This means that $R$ consists precisely of the elements
\begin{equation*}
	p(({x},{x}),({x},{y}),({y},{y}))\in \mathsf{F}(x,y)\times\mathsf{F}(x,y)
\end{equation*}
where $p\in \mathsf{F}(x,y,z)$. Hence 
\begin{equation}\label{eq:relationR}
	R=\{(p(x,x,y),p(x,y,y))\mid p\in\mathsf{F}(x,y,z)\}\subseteq \mathsf{F}(x,y)\times\mathsf{F}(x,y).
\end{equation}

\begin{lemma}\label{lem:Mal'tsevif and only ifRcontains(y,x)}
The variety $\mathbb{V}$ is Mal'tsev if and only if the relation $R$ contains $(y,x)$.
\end{lemma}

\begin{proof}
"$\Rightarrow$":  Let $\mathbb{V}$ be a Mal'tsev variety. We see from Equation~\eqref{eq:relationR} that $R$ is reflexive since we can write
\begin{equation*}
	(t(x,y),t(x,y))=(p(x,x,y),p(x,y,y))
\end{equation*}
for any binary term $t(x,y)\in\mathsf{F}(x,y)$ if we set $p(x,y,z) := t(x,z)$.  Hence $R$ is symmetric and $(y,x)\in R$.

"$\Leftarrow$": Let $s\colon S\rightarrowtail A\times A$ be a reflexive relation on an algebra~$A$. We show that if $(a,b)\in S$ then also $(b,a)\in S$. For this let $f\colon \mathsf{F}(x,y)\to A$ be the unique morphism such that $f(x)=a$ and $f(y)=b$. It is easy to see that there exists a morphism $\varphi\colon R\to S$ such that the diagram
\begin{equation}\label{diagr:relationR}
	\begin{tikzcd}
		R && {\mathsf{F}(x,y)\times \mathsf{F}(x,y)} \\
		\\
		S && {A\times A,}
		\arrow["{r}", tail, from=1-1, to=1-3]
		\arrow["{f\times f}", from=1-3, to=3-3]
		\arrow["{\varphi}"', dotted, from=1-1, to=3-1]
		\arrow["{s}"', tail, from=3-1, to=3-3]
	\end{tikzcd}
\end{equation}
commutes since $R$ is the smallest relation which contains $(x,x)$, $(x,y)$ and $(y,y)$ and $(f\times f)(x,x)=(a,a)$, $(f\times f)(x,y)=(a,b)$ and $(f\times f)(y,y)=(b,b)$ lie in~$S$. The assumption $(y,x)\in R$ implies that $(b,a)=(f\times f)(y,x)$ is contained in~$S$.
\end{proof}

\begin{lemma}\label{lem:Rcontains(y,x)if and only ifp}
The relation $R$ contains $(y,x)$ if and only if there exists a ternary term $p\in\mathsf{F}(x,y,z)$ such that $p(x,x,y)= y$ and $p(x,y,y)= x$.
\end{lemma}

\begin{proof}
This is clear by Equation~\eqref{eq:relationR}.
\end{proof}

To characterize the varieties which are $\mathcal{M}_{\text{reg}}$-Mal'tsev, we consider the smallest \emph{regular} relation on $\mathsf{F}(x,y)$ which contains $(x,x)$, $(x,y)$ and $(y,y)$. This relation can be obtained by first constructing the cokernel pair $(Q',q'_1,q'_2)$ of the relation~$R$. The equalizer $(E,e)$ of $q'_1,q'_2$ yields the desired regular relation and we get a unique morphism $j\colon R\to E$ such that $ej=r$.

\begin{equation*}
	\begin{tikzcd}
		& E \\
		R && {\mathsf{F}(x,y)\times\mathsf{F}(x,y)} && Q'
		\arrow["r", tail, from=2-1, to=2-3]
		\arrow["{q'_2}"', shift right=1, from=2-3, to=2-5]
		\arrow["{q'_1}", shift left=1, from=2-3, to=2-5]
		\arrow["e", tail, from=1-2, to=2-3]
		\arrow["j", tail, from=2-1, to=1-2]
	\end{tikzcd}
\end{equation*}

\begin{lemma}\label{lem:wMif and only ifE}
A variety $\mathbb{V}$ is $\mathcal{M}_{\text{reg}}$-Mal'tsev if and only if $E$ contains $(y,x)$.
\end{lemma}

\begin{proof}
"$\Rightarrow$": By definition of an $\mathcal{M}_{\text{reg}}$-Mal'tsev category, the reflexive regular relation $E$ is symmetric. Hence $(y,x)\in E$.

"$\Leftarrow$": Let $A$ be an algebra in $\mathbb{V}$ and $s\colon S\rightarrowtail A\times A$ be a reflexive regular relation on~$A$. We show that if $(a,b)\in S$ then also $(b,a)\in S$. We consider the extended version
\begin{equation*}
	\xymatrix{& E \ar@{ >->}[rdd]^-{e} \ar@{.>}@/^/[ddd]^-{\omega} \\
		\\
		R \ar@{ >->}[ruu]^-{j} \ar@{ >->}[rr]_(.3){r} \ar[dd]_-{\varphi} && \mathsf{F}(x,y)\times \mathsf{F}(x,y) \ar[dd]^-{f\times f} \ar@<2pt>[rr]^-{q'_1} \ar@<-2pt>[rr]_-{q'_2} && Q' \ar@{.>}[dd]^-{\psi} \\
		& S \ar@{ >->}[rd]^-{s} \\
		S \ar@{ >->}[rr]_-{s} \ar@{-}@<1pt>[ru] \ar@{-}@<-1pt>[ru] && A\times A \ar@<2pt>[rr]^-{q''_1} \ar@<-2pt>[rr]_-{q''_2} && {Q''}}
\end{equation*}
of Diagram~\eqref{diagr:relationR}, where $(Q'',q''_1,q''_2)$ is the cokernel pair of $s$ and, by assumption, $s$ is the equalizer of $q''_1,q''_2$. Furthermore, $\psi$ is the unique map such that the upper and lower right-hand side squares of the diagram commute. Then
\begin{equation*}
	q''_1(f\times f)e
	=\psi q'_1 e
	=\psi q'_2 e
	=q''_2 (f\times f) e,
\end{equation*} 
and there exists a unique map $\omega\colon E\to S$ such that $s\omega=(f\times f)e$ . Since $(y,x)\in E $ by assumption and $(f\times f)({y},{x})=(b,a)$, it follows that $(b,a)\in S$.
\end{proof}

\begin{theorem}\label{thm:syntaxwM}
A finitary one-sorted variety $\mathbb{V}$ of universal algebras is an $\mathcal{M}_{\text{reg}}$-Mal'tsev category if and only if there exist integers $k,m,N\geqslant 0$, binary terms $f_1,g_1,\ldots,f_k,g_k\in \mathsf{F}(x,y)$, ternary terms $p_1,\ldots,p_m\in\mathsf{F}(x,y,z)$, $(2(k+2m+1))$-ary terms $s_1,\ldots,s_N$, ${(2(k+m+2))}$-ary terms $\sigma_1,\ldots,\sigma_{N+1}$ and, for all $i\in\{1,\ldots,N+1\}$, $(k+m+1)$-ary terms $\eta^{(i)}_1,\eta^{(i)}_2,\epsilon^{(i)}_1,\epsilon^{(i)}_2$ such that the following identities (on variables $x$, $y$, $u$, $u'$, $v_1,\dots,v_k$, $v'_1,\dots,v'_k$, $w_1,\dots,w_m$ and $w'_1,\dots,w'_m$) hold, where we write $\vec{v}$ for $v_1,\ldots,v_k$ and $\vec{w}$ for $w_1,\ldots,w_m$, and analogously for $\vec{v}'$ and~$\vec{w}'$:
{\allowdisplaybreaks
	\begin{subequations}\label{subeq:wMeta}
		\begin{alignat}{2}
			&\eta^{(i)}_\alpha(y,f_1(x,y),\ldots,f_k(x,y),p_1(x,x,y),\ldots,p_m(x,x,y))\nonumber \\*
			=&\epsilon^{(i)}_\alpha(y,f_1(x,y),\ldots,f_k(x,y),p_1(x,x,y),\ldots,p_m(x,x,y)),\label{eq:wMetay}\\[3ex]
			&\eta^{(i)}_\alpha(x,g_1(x,y),\ldots,g_k(x,y),p_1(x,y,y),\ldots,p_m(x,y,y))\nonumber \\*
			=&\epsilon^{(i)}_\alpha(x,g_1(x,y),\ldots,g_k(x,y),p_1(x,y,y),\ldots,p_m(x,y,y)),\label{eq:wMetax}
		\end{alignat}
\end{subequations}}%
for all $i\in\{1,\ldots,N+1\}$ and $\alpha\in\{1,2\}$;
{\allowdisplaybreaks
	\begin{subequations}
		\begin{alignat}{2}
			&\sigma_{i}(u,\vec{v},\vec{w},u',\vec{v}',\vec{w}',\epsilon^{(i)}_1(u,\vec{v},\vec{w}),\epsilon^{(i)}_2(u',\vec{v}',\vec{w}'))\nonumber\\*
			=&s_{i}(u,\vec{v},\vec{w},\vec{w},u',\vec{v}',\vec{w}',\vec{w}'),\label{eq:wModd}\\[3ex]
			&s_{i}(u,\vec{v},\vec{w},\vec{w}',u',\vec{v}',\vec{w}',\vec{w})\nonumber \\*
			=&\sigma_{i+1}(u,\vec{v},\vec{w},u',\vec{v}',\vec{w}',{\eta}^{(i+1)}_1(u,\vec{v},\vec{w}),{\eta}^{(i+1)}_2(u',\vec{v}',\vec{w}')),\label{eq:wMeven}
		\end{alignat}
\end{subequations}}%
for all $i\in\{1,\ldots,N\}$ and
\begin{subequations}\label{subeq:wMuu'}
	\begin{alignat}{2}
		u
		&=\sigma_1(u,\vec{v},\vec{w},u',\vec{v}',\vec{w}',\eta^{(1)}_1(u,\vec{v},\vec{w}),\eta^{(1)}_2(u',\vec{v}',\vec{w}')),\label{eq:wMu}\\
		u'
		&=\sigma_{N+1}(u,\vec{v},\vec{w},u',\vec{v}',\vec{w}',\epsilon^{(N+1)}_1(u,\vec{v},\vec{w}),\epsilon^{(N+1)}_2(u',\vec{v}',\vec{w}'))\label{eq:wMu'}.
	\end{alignat}
\end{subequations}
\end{theorem}

\begin{proof}
By Lemma~\ref{lem:wMif and only ifE}, we know that $\mathbb{V}$ is $\mathcal{M}_{\text{reg}}$-Mal'tsev if and only if $(y,x)\in E$. To construct~$E$, we first build the cokernel pair $(Q',q'_1,q'_2)$ by means of the coequalizer $q'$ of the maps $\iota'_1r,\iota'_2r$, where $\iota'_1,\iota'_2$ are the coproduct inclusions into $\mathsf{F}(x,y)^2+\mathsf{F}(x,y)^2$, as in the following diagram:
\begin{equation*}
	\xymatrix{
		R \ar@{ >->}[rr]^-{r} \ar@{ >->}[dd]_-{r} && \mathsf{F}(x,y)\times \mathsf{F}(x,y) \ar[dd]^-{\iota'_2} \ar@/^1.4pc/[rddd]^-{q'_2} & \\
		&&& \\
		\mathsf{F}(x,y)\times \mathsf{F}(x,y) \ar[rr]_-{\iota'_1} \ar@<-2pt>@/_1.4pc/[rrrd]_-{q'_1} && (\mathsf{F}(x,y)\times \mathsf{F}(x,y))+(\mathsf{F}(x,y)\times \mathsf{F}(x,y)) \ar[rd]_-{q'} & \\
		&&& Q'}
\end{equation*}
The algebra $Q'$ is given by the quotient of $\mathsf{F}(x,y)^2+\mathsf{F}(x,y)^2$, with respect to, due to Equation~\eqref{eq:relationR}, the smallest congruence $C'\subseteq(\mathsf{F}(x,y)^2+\mathsf{F}(x,y)^2)^2$ containing all pairs of the form
\begin{equation*}
	\Bigl(\iota'_1\bigl({p(x,x,y)},{p(x,y,y)}\bigr),\iota'_2\bigl(p(x,x,y),p(x,y,y)\bigr)\Bigr),
\end{equation*}
where $p\in\mathsf{F}(x,y,z)$. The algebra $E$ is then given by all pairs $(s,t)\in \mathsf{F}(x,y)\times \mathsf{F}(x,y)$ such that $q'(\iota'_1(s,t))=q'(\iota'_2(s,t))$. Thus, the condition $(y,x)\in E$ is equivalent to $(\iota'_1(y,x),\iota'_2(y,x))\in C'$. Proceeding from here analogously as in the proof of Theorem~\ref{thm:syntaxWM} yields the claim. We note that, in comparison to the proof of Theorem~\ref{thm:syntaxWM}, here $P=\{(s(x,y),t(x,y))\in\mathsf{F}(x,y)\times\mathsf{F}(x,y)\mid s(x,x)=t(x,x)\}$ is replaced by $\mathsf{F}(x,y)\times\mathsf{F}(x,y)$, which explains why the condition $f_i(x,x)=g_i(x,x)$ for all $i\in\{1,\ldots,k\}$ does not appear here.
\end{proof}

\end{document}